\documentclass[11pt]{article}
\usepackage{amsmath}
\usepackage{amsfonts}
\usepackage{amssymb}
\usepackage{mathrsfs}
\usepackage{amsthm}
\usepackage{palatino}
\usepackage[margin=2.5cm, vmargin={1.5cm}]{geometry}

\usepackage{dcolumn}

\usepackage{times}
\usepackage{graphicx}
\usepackage{color}

\usepackage{pstricks}
\usepackage{pst-plot}
\usepackage{pst-grad}

\newrgbcolor{LemonChiffon}{1.0 0.98 0.8}
\newrgbcolor{magenta}{1.0 0.3 0.4}
\newrgbcolor{mistyrose}{1.0 0.9 0.8}
\newrgbcolor{deeppink}{1.0 0.08 0.58}
\newrgbcolor{lightsalmon}{1.0 0.63 0.48}
\newrgbcolor{lightred}{0.7 0.0 0.0}
\newrgbcolor{lightblue}{0.8 0.8 1.0}
\newrgbcolor{indianred}{0.80  0.36  0.36}
\newrgbcolor{byellow}{0.93 0.76 0.3}
\newrgbcolor{lightorange}{1 0.9 0.8}

\renewcommand\footnotemark{}

\begin{document}

\title{Formulas for non-holomorphic Eisenstein series and for the Riemann zeta function at odd integers}

\author{
  Cormac ~O'Sullivan\footnote{{\it Date:} Aug 16, 2018.
\newline \indent \ \ \
  {\it 2010 Mathematics Subject Classification:} 11F30, 11F37, 11M06.
  \newline \indent \ \ \
Support for this project was provided by a PSC-CUNY Award, jointly funded by The Professional Staff Congress and The City
\newline \indent \ \ \
University of New York.}
  }

\date{}

\maketitle

\def\s#1#2{\langle \,#1 , #2 \,\rangle}

\def\H{{\mathbb{H}}}
\def\F{{\frak F}}
\def\C{{\mathbb C}}
\def\R{{\mathbb R}}
\def\Z{{\mathbb Z}}
\def\Q{{\mathbb Q}}
\def\N{{\mathbb N}}
\def\ei{{\mathbb E}}
\def\es{{\text{\boldmath $E$}}}
\def\ai{{\mathbb A}}
\def\bl{{\mathbb L}}
\def\G{{\Gamma}}
\def\GH{{\G \backslash \H}}
\def\g{{\gamma}}
\def\L{{\Lambda}}
\def\ee{{\varepsilon}}
\def\K{{\mathcal K}}
\def\Re{\mathrm{Re}}
\def\Im{\mathrm{Im}}
\def\PSL{\mathrm{PSL}}
\def\SL{\mathrm{SL}}
\def\Vol{\operatorname{Vol}}
\def\lqs{\leqslant}
\def\gqs{\geqslant}
\def\sgn{\operatorname{sgn}}
\def\res{\operatornamewithlimits{Res}}
\def\li{\operatorname{Li_2}}
\def\lip{\operatorname{Li}'_2}
\def\pl{\operatorname{Li}}
\def\nb{{\mathcal B}}
\def\cc{{\mathcal C}}
\def\nd{{\mathcal D}}
\def\dd{\displaystyle}

\def\clp{\operatorname{Cl}'_2}
\def\clpp{\operatorname{Cl}''_2}
\def\farey{\mathscr F}

\def\ca{{\frak a}}
\def\cb{{\frak b}}
\def\cc{{\frak c}}
\def\cd{{\frak d}}
\def\ci{{\infty}}

\def\sa{{\sigma_\frak a}}
\def\sb{{\sigma_\frak b}}
\def\sc{{\sigma_\frak c}}
\def\sd{{\sigma_\frak d}}
\def\si{{\sigma_\infty}}

\newcommand{\stira}[2]{{\genfrac{[}{]}{0pt}{}{#1}{#2}}}
\newcommand{\stirb}[2]{{\genfrac{\{}{\}}{0pt}{}{#1}{#2}}}
\newcommand{\norm}[1]{\left\lVert #1 \right\rVert}

\newcommand{\e}{\eqref}


\newtheorem{theorem}{Theorem}[section]
\newtheorem{lemma}[theorem]{Lemma}
\newtheorem{prop}[theorem]{Proposition}
\newtheorem{conj}[theorem]{Conjecture}
\newtheorem{cor}[theorem]{Corollary}
\newtheorem{assume}[theorem]{Assumptions}

\newcounter{coundef}
\newtheorem{adef}[coundef]{Definition}

\newcounter{counrem}
\newtheorem{remark}[counrem]{Remark}

\renewcommand{\labelenumi}{(\roman{enumi})}
\newcommand{\spr}[2]{\sideset{}{_{#2}^{-1}}{\textstyle \prod}({#1})}
\newcommand{\spn}[2]{\sideset{}{_{#2}}{\textstyle \prod}({#1})}

\numberwithin{equation}{section}

\bibliographystyle{alpha}

\begin{abstract}
New expressions are given for the Fourier expansions of non-holomorphic  Eisenstein series with weight $k$. Among other applications, this leads to  non-holomorphic analogs of  formulas of Ramanujan, Grosswald and Berndt containing Eichler integrals of holomorphic Eisenstein series.
\end{abstract}

\section{Introduction}
\subsection{Eisenstein series}
Let $\G =\SL_2(\Z)$ act on the upper half plane $\H$ in the usual way, with $\G_\ci=\left\{ \pm (\smallmatrix 1
& n \\ 0 & 1 \endsmallmatrix ) : n\in \Z\right\}$ the subgroup of translations fixing $\ci$. Throughout it is assumed that  $z=x+iy \in \H$. The holomorphic Eisenstein series are basic modular forms with definition
\begin{equation}\label{esd}
E_k(z):= \sum_{\g \in \G_\ci \backslash \G} \frac{1}{j(\g, z)^k} = \frac{1}{2} \sum_{\substack{c,d \in \Z \\ \gcd(c,d)=1}} \frac{1}{(cz+d)^k}
\end{equation}
for even $k \gqs 4$ as in for example \cite[(9) p. 13]{Za}, where  $j((\smallmatrix a
& b \\ c & d \endsmallmatrix ),z):=cz+d$. They have weight $k$, meaning that
\begin{equation}\label{wtkh}
E_k(\g z) = j(\g, z)^{k} E_k(z)
\end{equation}
for all $\g \in \G$.
Their Fourier expansions are given by
\begin{equation}\label{eisk}
E_k(z)=1+\frac{2}{\zeta(1-k)} \sum_{m=1}^\infty \sigma_{k-1}(m) e^{2\pi i m z}
\end{equation}
with  $\zeta(s)$ the Riemann zeta function and $\sigma_s(m):=\sum_{d|m} d^s$ the divisor power function.
The Fourier coefficients of $E_k(z)$ are rational by the relation $2/\zeta(1-k)=-2k/B_k$, with $B_k$ indicating the $k$th Bernoulli number. Since it is always convergent,  \e{eisk}  may be used to extend the definition of $E_k(z)$ to all $k \in 2\Z$ (and indeed to all $k\in \C$). However we cannot expect \e{wtkh} will  continue to hold.

Maass introduced and developed a similar kind of Eisenstein series that is not holomorphic but is instead an eigenfunction of the hyperbolic Laplacian and hence real analytic on $\H$.
To introduce its most symmetric form, for each $k \in 2\Z$ we set
\begin{equation}\label{defe}
E_k(z,s):= \sum_{\g \in \G_\ci \backslash \G} \left(\frac{j(\g, z)}{|j(\g, z)|}\right)^{-k} \Im (\g z)^s \qquad (\Re(s)>1).
\end{equation}
These series satisfy
\begin{equation}\label{wtk}
E_k(\g z,s) = \left(\frac{j(\g, z)}{|j(\g, z)|}\right)^{k} E_k(z,s)
\end{equation}
for all $\g \in \G$, and this property  may be described as transforming with {\em non-holomorphic weight $k$} to distinguish it from the holomorphic weight in \e{wtkh}. With the identity $\Im(\g z)=y/|j(\g,z)|^2$ we
clearly have
\begin{equation}\label{whp}
y^{-k/2} E_k(z,s)= \sum_{\g \in \G_\ci \backslash \G} \frac{\Im (\g z)^{s-k/2}}{j(\g, z)^k}
\end{equation}
and the holomorphic Eisenstein series is recovered from \e{whp} when $s=k/2$ for even $k \gqs 4$:
\begin{equation}\label{ehol}
E_k(z)=y^{-k/2} E_k(z,k/2).
\end{equation}
In general, a function transforming with non-holomorphic weight $k$ may be converted into one with holomorphic weight $k$ by multiplying it by $y^{-k/2}$; for the other direction multiply by $y^{k/2}$.

The reason we restrict to $k$ even in \e{esd} and \e{defe} is that the sign of $j(\g,z)^k$ is not well-defined for $\g\in \G_\ci \backslash \G$ when $k$ is odd. To make \e{esd} and \e{defe} well-defined we could quotient by $B:=\left\{(\smallmatrix 1
& n \\ 0 & 1 \endsmallmatrix ) : n\in \Z\right\}$ instead of $\G_\ci$ (as in the series on the right of \e{esd}) but then everything cancels when $k$ is odd since
 $(\smallmatrix -1
& 0 \\ 0 & -1 \endsmallmatrix ) \in \G
$. Taking $k$ odd, quotienting by $B$ and restricting to matrix elements with positive bottom left entry gives a non-zero result in the holomorphic case, with a Fourier expansion similar to \e{eisk}, but it does not transform with weight $k$. See \e{dbleis3} for a similar construction that does transform correctly.

As in \cite[Sect. 3.4]{IwSp} set
$$
\theta(s) := \pi^{-s} \G(s) \zeta(2s).
$$
With this notation, the functional equation for the Riemann zeta function becomes
\begin{equation}\label{riez}
 \theta((1-s)/2) = \theta(s/2)
\end{equation}
and if we let $\mathcal Z$ be the set of non-trivial zeros of $\zeta(s)$, then $\theta(s)$ has its zeros exactly in  $\mathcal Z/2$. Also $\theta(s)$ has only two poles; they are simple and at $s=0,1/2$ with residues $-1/2$ and $1/2$ respectively.
Next, for $k \in 2\Z$, put
\begin{align}
  \theta_k(s)  & := \theta(s) \cdot s(s+1) \cdots (s+|k|/2 -1) \label{pol} \\
   & \phantom{:}= \pi^{-s} \G(s+|k|/2) \zeta(2s) \notag
\end{align}
and define the {\em completed non-holomorphic Eisenstein series of weight $k$} as
\begin{equation}\label{compl}
E^*_k(z,s):= \theta_k(s) E_k(z,s).
\end{equation}

This is our main object of study.
Non-holomorphic Eisenstein series are fundamental modular forms with many connections in the literature and we mention a small sample here. In weight $0$, $E^*_{0}(z,s)$ is associated with the   Epstein zeta function and evaluation of the second term in the Laurent expansion of $E^*_{0}(z,s)$ at $s=0$ or $1$ gives the  Kronecker limit formula \cite{DIT} with important number theoretic applications. Higher terms in the Laurent expansion of $E^*_{k}(z,s)$ at $s=k/2$ are used in \cite{LR16} to give a basis for  polyharmonic Maass forms. The Rankin-Selberg method \cite[Sect. 1.6]{Bu} gives the convolution $L$-function for two cusp forms and requires $E^*_{0}(z,s)$ if the cusp forms have the same weight and $E^*_{k}(z,s)$ if the difference of their weights is $k$. The continuous spectrum for the weight $k$ hyperbolic Laplacian is constructed with Eisenstein series; see \cite{IwSp} for weight $0$ and \cite[Sect. 4]{DFI02} for general weights. Recent  interest in the series $E^*_{k}(z,s)$ comes with their connection to modular graph functions \cite{B18,DD}. In this paper we focus on applications to evaluating the Riemann zeta function and automorphic $L$-functions at integers.

\subsection{Main results}

For the Fourier expansion of $E^*_0(z,s)$ we need the next definitions. Let
\begin{equation} \label{bessel}
  K_w(y):=\frac 12\int_0^\infty e^{-y(t+1/t)/2} t^{w-1}\, dt \qquad (w\in \C, \ y>0)
\end{equation}
be the modified Bessel function. It is entire in the parameter $w$  and,
following (1.26), (1.27) in \cite{IwSp}, a convenient Whittaker function variant may be defined as
\begin{equation}\label{whit}
W_s(z) := 2|y|^{1/2} K_{s-1/2}(2\pi |y|)\cdot e^{2\pi i x}
\end{equation}
for $s \in \C$ and $z \in \C-\R$. This gives $W_s(z) = W_s(\overline{z})$ and
\begin{equation*}
  W_s(z) \sim e^{2\pi i z} \quad \text{as}  \quad  y\to \infty.
\end{equation*}
As seen in \cite[Sect. 1.6]{Bu} or \cite[Sect. 3.4]{IwSp} for example, the Fourier expansion of the weight $0$ Eisenstein series for the modular group is given by
\begin{equation}\label{e0}
E^*_{0}(z,s) = \theta(s) y^s+ \theta(1-s) y^{1-s} + \sum_{m \in \Z_{\neq 0}} \frac{ \sigma_{2s-1}(|m|)}{|m|^s} W_s(m z).
\end{equation}
Since all the terms on the right are defined for $s\in \C$, and the function $W_s(z)$ has exponential decay, \e{e0} furnishes us with the meromorphic continuation of $E^*_{0}(z,s)$ to all $s\in \C$. The only poles are contributed by the constant (with respect to $x$) term $\theta(s) y^s+ \theta(1-s) y^{1-s}$. They are at $s=0,$ $1$ and simple with
\begin{equation} \label{simp}
  \res_{s=0} E^*_{0}(z,s) = -1/2, \qquad \res_{s=1} E^*_{0}(z,s) = 1/2.
\end{equation}

We will develop the theory of non-holomorphic, weight $k$ Eisenstein series using \e{e0} as the starting point and our first result is essentially obtained  by differentiating  \e{e0}.
For $n,$ $r\in \Z$  define the  polynomial $P^n_r(x) \in \Z[x]$ as follows. When $n\gqs 0$ set
\begin{equation}\label{pnr}
P^n_r(x) := \sum_{\ell=|r|}^n \frac{(2n)! }{(n-\ell)! (\ell+r)! (\ell-r)!} (-x)^\ell,
\end{equation}
giving a variant of the generalized Laguerre polynomial.
When $n<0$ define $P_r^{n}(x)$  to be $(-1)^r P_r^{-n}(-x)$.

\begin{theorem} \label{allk2}
For all  $k \in 2\Z$  and all $s\in \C$ with $\Re(s)>1$,
\begin{multline} \label{ekzs2b}
 E^*_k(z,s)= \theta_k(s) y^s+\theta_k(1-s) y^{1-s} \\+ 2^{-|k|}
   \sum_{m \in \Z_{\neq 0}} \frac{ \sigma_{2s-1}(|m|)}{|m|^s}  \sum_{r=-|k|/2}^{|k|/2} \left(\frac{m}{|m|}\right)^r  P^{k/2}_r(4\pi my) \cdot W_{s+r}(mz).
\end{multline}
The  right hand side of \e{ekzs2b} converges absolutely and uniformly on compacta for all $s\in \C$, giving
the meromorphic continuation of $E^*_k(z,s)$ to the whole $s$ plane. When $k$ is nonzero, $E^*_k(z,s)$ is an entire function of $s$.
\end{theorem}

The expansion \e{ekzs2b} for $E^*_k(z,s)$ appears to be new and  is more suited to our needs in this paper than the expansions in the literature in terms of  Whittaker functions $W_{\kappa,\mu}$ or other confluent hypergeometric functions; see the comparisons at the end of Section \ref{bounds}.

 The well-known functional equation for this Eisenstein series takes the form
\begin{equation}\label{fek}
  E^*_k(z,1-s) =  E^*_k(z,s)
\end{equation}
and is evident from \e{ekzs2b} and the identities
\begin{equation*}
  \sigma_w(m)=m^w \sigma_{-w}(m), \qquad K_w(y)=K_{-w}(y).
\end{equation*}
The relation
\begin{equation}\label{connj}
E^*_{-k}(z,s) = \overline{E^*_k(z,\overline{s})}
\end{equation}
may also be verified using  \e{ekzs2b}, or more simply with \e{defe} and   analytic continuation to all $s\in \C$.

For $k \in 2\Z$, $h \in \Z$ and $u \in \Z_{\geqslant 0}$, define the integers
\begin{equation}\label{hstar}
  h^*:=|h-1/2| -1/2 = \begin{cases} h-1 & \text{ if \ } h \geqslant 1; \\
-h & \text{ if \ } h \leqslant 0,
\end{cases}
\end{equation}
and
\begin{equation}\label{defak2}
\mathcal A^k_h(u) := \begin{cases}
 \displaystyle (-1)^{u+k/2} u! \binom{k/2+h-1}{u} \binom{k/2-h}{u} \phantom{\Bigg |} & \text{ \ if \ } k\gqs 0;\\
 \displaystyle (-1)^u \frac{[(u-k/2)!]^2}{u!} \binom{h-1}{u-k/2} \binom{-h}{u-k/2} \phantom{\Bigg |}  & \text{ \ if \ } k \lqs 0.
\end{cases}
\end{equation}
Note that throughout this paper we use the usual generalized binomial coefficients as given in \e{bin}.
With this notation we may describe the Fourier development of $E^*_k(z,s)$ very explicitly when $s$ is an integer:

\begin{theorem}\label{main2}
For all  $h \in \Z$
and  $k \in 2\Z$, except $(h,k)=(0,0)$ or $(1,0)$,
\begin{multline} \label{mnk2}
   E^*_k(z,h)= \theta_k(h) y^h+\theta_k(1-h) y^{1-h} +
   \sum_{m =1}^\infty \frac{ \sigma_{2h-1}(m)}{m^h} \ e^{2\pi i m z} \sum_{u=0}^{h^* +k/2}  \mathcal A^k_h(u) \cdot (4\pi m y)^{-u+k/2}\\
+
   \sum_{m =1}^\infty \frac{ \sigma_{2h-1}(m)}{m^h} \ e^{-2\pi i m \overline{z}} \sum_{u=0}^{h^*-k/2}  \mathcal A^{-k}_h(u) \cdot (4\pi m y)^{-u-k/2}.
\end{multline}
Since some of the $\mathcal A$ coefficients are zero, the upper bounds for the indices of the inner sums in \e{mnk2} may be reduced exactly in the following  cases. If $h^*<k/2$
then the upper bound of $h^*+k/2$ for the first inner sum may be reduced to  $k/2-1-h^*$. 
If $h^*<-k/2$  then the upper bound of $h^*-k/2$ for the second inner sum may be reduced to  $-k/2-1-h^*$.
\end{theorem}


\SpecialCoor
\psset{griddots=5,subgriddiv=0,gridlabels=0pt}
\psset{xunit=1.2cm, yunit=0.8cm, runit=1.2cm}
\psset{linewidth=1pt}
\psset{dotsize=3.5pt 0,dotstyle=*}
\psset{arrowscale=1.5,arrowinset=0.3}

\begin{figure}[ht]
\centering
\begin{pspicture}(-5,-5)(6,6) 

\pscustom[fillstyle=gradient,gradangle=0,linecolor=white,gradmidpoint=1,gradbegin=white,gradend=lightorange,gradlines=100]{%
\pspolygon*[linecolor=lightblue](-3.5,4.5)(0.5,0.5)(4.5,4.5)} 
\pscustom[fillstyle=gradient,gradangle=180,linecolor=white,gradmidpoint=1,gradbegin=white,gradend=lightorange,gradlines=100]{%
\pspolygon*[linecolor=lightblue](-3.5,-4.5)(0.5,-0.5)(4.5,-4.5)} 
\pscustom[fillstyle=gradient,gradangle=270,linecolor=white,gradmidpoint=1,gradbegin=white,gradend=lightblue,gradlines=100]{%
\pspolygon*[linecolor=lightblue](4.5,3.5)(1,0)(4.5,-3.5)} 
\pscustom[fillstyle=gradient,gradangle=90,linecolor=white,gradmidpoint=1,gradbegin=white,gradend=lightblue,gradlines=100]{%
  \pspolygon[linecolor=lightblue](-3.5,3.5)(0,0)(-3.5,-3.5)} 

\psline[linecolor=gray]{->}(-4,0)(5.5,0)
\psline[linecolor=gray]{->}(0,-5)(0,5.5)
\psline[linecolor=red,linestyle=dotted,dotsep=1pt](0.5,-5)(0.5,5.5)

\psarc[linecolor=red]{<->}(0.5,1.4){2}{75}{105}
\psarc[linecolor=red]{<->}(2.5,0){2}{-15}{15}

\multirput(-3,-0.2)(1,0){8}{\psline[linecolor=gray](0,0)(0,0.4)}
\multirput(-0.07,-4.5)(0,0.5){19}{\psline[linecolor=gray](0,0)(0.14,0)}

\psline[linecolor=red](-4,-5)(5,4)
\psline[linecolor=red](-4,5)(5,-4)
\psline[linecolor=red](-4,4)(5,-5)
\psline(-4,-4)(5,5)

\multirput(-3,-4)(1,0){8}{\psdot(0,0)}
\multirput(-3,-3)(1,0){8}{\psdot(0,0)}
\multirput(-3,-2)(1,0){8}{\psdot(0,0)}
\multirput(-3,-1)(1,0){8}{\psdot(0,0)}
\multirput(-3,-0)(1,0){8}{\psdot(0,0)}
\multirput(-3,1)(1,0){8}{\psdot(0,0)}
\multirput(-3,2)(1,0){8}{\psdot(0,0)}
\multirput(-3,3)(1,0){8}{\psdot(0,0)}
\multirput(-3,4)(1,0){8}{\psdot(0,0)}

\pscircle*[linecolor=white,linewidth=1pt](0,0){0.08}
\pscircle[linecolor=black,linewidth=1pt](0,0){0.08}
\pscircle*[linecolor=white,linewidth=1pt](1,0){0.08}
\pscircle[linecolor=black,linewidth=1pt](1,0){0.08}

\rput(5.1,-0.4){$h$}
\rput(-0.3,5){$k$}
\rput(4.7,2){$E^*_k(z,h)$}
\rput(5.8,5){$h=k/2$}
\rput(6.1,4){$h=k/2+1$}

\rput(1.8,3.5){$h^*<k/2$}
\rput(1.8,-3.5){$h^*<-k/2$}
\rput(3,0.5){$|k/2|<h$}
\rput(-2,0.5){$|k/2|<1-h$}

\psline[linecolor=gray]{->}(-4.5,3)(-4.5,4)
\psline[linecolor=gray]{->}(-4.5,2)(-4.5,1)
\rput(-4.8,3.5){$R$}
\rput(-4.8,1.5){$L$}


\end{pspicture}
\caption{The lattice of Eisenstein series in the $hk$ plane}
\label{bfig}
\end{figure}


All the terms in \e{mnk2} are  simple to calculate. For the constant term coefficients we have the following useful formula which derives from the basic properties of $\zeta(s)$ and $\G(s)$.
For all $k,m \in \Z$ with $k$ even and $(k,m) \neq (0,0)$,
\begin{equation}\label{thet}
 \theta_{k}(m)=\begin{cases}
 \pi^{-m} (m+|k|/2-1)! \zeta(2m) & \quad \text{if} \quad 0 \lqs m; \\
 0 & \quad \text{if} \quad -|k|/2 < m < 0; \\
 \displaystyle (-1)^{k/2}(4\pi)^m \frac{(2|m|)!}{(|m|-|k|/2)!} \zeta(1-2m) & \quad \text{if} \quad m \lqs -|k|/2.
 \end{cases}
\end{equation}

Figure \ref{bfig} illustrates the lattice of Eisenstein series described in Theorem \ref{main2}. Each dot in position $(h,k)$ represents the one dimensional $\C$-vector space  generated by $E^*_k(z,h)$. The lattice naturally breaks into four triangular regions. In the upper triangle we have $h^*<k/2$, and by \e{mnk2} the Eisenstein series here are exactly those with no negative terms in their Fourier expansions (i.e. no terms containing $e^{2\pi i m \overline z}$ with $m$ negative). In the lower triangle we have $h^*<-k/2$, and these  series  have no positive terms (i.e. no terms containing $e^{2\pi i m  z}$ with $m$ positive). These upper and lower triangles are interchanged by conjugation. This reflective symmetry $k \leftrightarrow -k$ is indicated with arrows in the figure and comes from \e{connj}. The left-right symmetry  $h \leftrightarrow 1-h$ from \e{fek} is also indicated. The main diagonal line $h=k/2$ is shown along with its three images under these symmetries. The Eisenstein series corresponding to points on these lines are studied in Section  \ref{harmo}.

Theorem \ref{main2} was first stated in \cite[Thm. 3.1]{DO10} 
and the proof, which we give here in full, briefly sketched. Theorem  \ref{main2} was used there in providing new proofs of Manin's Periods Theorem and results of Kohnen and Zagier. We summarize some of these ideas in Sections \ref{inner} and \ref{proj}. Almost all of the methods and results of \cite{DO10} are contained in Chapter $12$ of \cite{CS17}, as acknowledged in that book's online errata.  Theorem \ref{main2} appears there in Section 12.2. We note that Theorem \ref{main2} must also be equivalent to  the many cases in Corollaries 2.4, 2.5 and 2.7 of \cite{KN09}. Recent computations of Brown in \cite{B18,B} are equivalent to Theorem \ref{main2}  when $|k/2|<h$ (the right triangular region in Figure \ref{bfig}) as described in Section \ref{brown}.

In Section \ref{torus}  we show that Theorem \ref{main2} may be used to prove and generalize  a result in \cite{CJK10} related to spectral zeta functions associated to a torus.   A formula of Terras and Grosswald on values of the Riemann zeta function at odd integers is also proven and generalized using  Theorem \ref{main2} in Section  \ref{terras}.
Terras \cite{Te76} and Grosswald \cite{Gr72} both mention the  inspiring  earlier formula of Ramanujan concerning $\zeta(2h-1)$. This is entry 21(i) of Chapter 14 in the second notebook; see \cite[p. 276]{Be89}. We refer the reader to \cite{Be77} and \cite{BS17} for a detailed account of this formula, its history and the related work of many authors. It has been greatly generalized  by Grosswald, see for example
\cite[Sect. 4]{Mur11}, and in a different direction by Berndt in \cite[Thm. 2.1]{Be77}, giving the transformation formula of a very general type of Eisenstein series. A special case of Berndt's formula is \cite[Thm. 2.2]{Be77} and we state a slightly rearranged version of this as follows. For all $z\in \H$ and $k\in \Z$,
\begin{multline}\label{bern}
z^k \left( 1+(-1)^k\right) \sum_{\ell=1}^\infty \sigma_{k-1}(\ell) e^{2\pi i \ell z}
- \left( 1+(-1)^k\right) \sum_{\ell=1}^\infty \sigma_{k-1}(\ell) e^{2\pi i \ell (-1/z)} \\
  = (2\pi i)^{1-k} \sum_{\substack{u,v \in \Z_{\gqs 0} \\u+v=2-k}}  \frac{B_{u}}{u!}\frac{B_{v}}{v!} z^{1-v}
-
   \begin{cases}
     \pi i -\log z & \hbox{if $k=0$;} \\
     \left( z^k-(-1)^k\right)\zeta(1-k) & \hbox{if $k\neq 0$.}
   \end{cases}
\end{multline}
Note that the sum containing Bernoulli numbers on the right is empty and vanishes when $k>2$.
If $k$ is odd then the left side of \e{bern} disappears and, letting $k=1-m$, a short calculation results in
\begin{equation} \label{euler}
  \zeta(m)=
              \begin{cases}
                -(2\pi i)^{m}B_{m}/(2 \cdot m!) & \hbox{if $m\in 2\Z_{\gqs 0}$;} \\
                0 & \hbox{if $m\in 2\Z_{< 0}$,}
              \end{cases}
\end{equation}
 incorporating Euler's famous formula. Set
\begin{equation*}
   U_k(z):= \sum_{m=1}^\infty \sigma_{k-1}(m) e^{2\pi i m z}=\sum_{m=1}^\infty \frac{m^{k-1}}{e^{-2\pi i m z}-1}.
\end{equation*}
For $k\in 2\Z$, \e{bern} then becomes what we may call the {\em master formula}:
\begin{equation}\label{bern2}
2\left(z^k U_k(z)
- U_k(-1/z) \right)
  = (2\pi i)^{1-k} \sum_{\substack{u,v \in \Z_{\gqs 0} \\u+v=1-k/2}}  \frac{B_{2u}}{(2u)!}\frac{B_{2v}}{(2v)!} z^{1-2v}
+
   \begin{cases}
     -\pi i/2 +\log z & \hbox{if $k=0$;} \\
     \left( 1-z^k\right)\zeta(1-k) & \hbox{if $k\neq 0$.}
   \end{cases}
\end{equation}
This is equivalent to Ramanujan's formula for $k \neq 0$ and $z$ purely imaginary. The negative  even  $k$ cases of \e{bern2} first appeared in \cite[p. 11]{Gr70}.
The reader may follow in the footsteps of Ramanujan, Grosswald \cite{Gr70,Gr72} and Berndt \cite{Be77} by employing this master formula to produce elegant identities. For example, substituting $z=i$ and $k=-2$ gives
\begin{equation} \label{z3}
  \zeta(3)=\frac{7\pi^3}{180}-2\sum_{m=1}^\infty \sigma_{-3}(m) e^{-2\pi m}.
\end{equation}
Letting $z=i$ and $k=2-2h$  gives the general form, for all $h \in 2\Z$,
\begin{equation}\label{lerch}
  \zeta(2h-1) = -\frac{(2\pi)^{2h-1}}{2} \sum_{\substack{u,v \in \Z_{\gqs 0} \\u+v=h}} (-1)^u \frac{B_{2u}}{(2u)!}\frac{B_{2v}}{(2v)!} -2\sum_{m=1}^\infty \sigma_{1-2h}(m) e^{-2\pi m}
\end{equation}
which is originally due to Lerch \cite[p. 276]{Be89} for positive even $h$. Moreover, \e{bern2} may be used to study the algebraic nature of the odd zeta values, as shown in \cite{Mur11}.

We next describe a natural non-holomorphic counterpart to the master formula \e{bern2}. Set
\begin{equation} \label{jkz}
   V_k(z):= \sum_{m=1}^\infty \sigma_{k-1}(m) e^{-2\pi i m \overline{z}}\sum_{u=0}^{-k}  \frac{(4\pi m y)^{u}}{u!}.
\end{equation}
\begin{theorem} \label{jkt}
For all $k\in 2\Z$ and $z\in \H$ we have
\begin{multline}\label{bern2comp}
2\left(z^k V_k(z)
- V_k(-1/z) \right)
  =
\frac{2\zeta(2-k)}{(2\pi i)^k} \left( \frac y\pi\right)^{1-k}\left(|z|^{2k-2}-z^k \right)
\\
-(2\pi i)^{1-k} \sum_{\substack{u,v \in \Z_{\gqs 0} \\u+v=1-k/2}}  \frac{B_{2u}}{(2u)!}\frac{B_{2v}}{(2v)!} z^{1-2v}
+
   \begin{cases}
     0 & \hbox{if $k>0$;} \\
     \pi i/2 +\overline{\log z} & \hbox{if $k=0$;} \\
     \left(1- z^k\right)\zeta(1-k) & \hbox{if $k< 0$.}
   \end{cases}
\end{multline}
\end{theorem}
Theorem \ref{jkt} is proved, and some of its consequences explored,   in Section \ref{harmo}.  For example \e{lerch} has a companion identity and adding \e{lerch} to its companion shows, for even $h\gqs 2$,
\begin{equation}\label{lerchxxx}
  \zeta(2h-1) = \frac{4^{h-1}}{\pi}\zeta(2h) -\sum_{m=1}^\infty \sigma_{1-2h}(m)  e^{-2\pi m}\left(1+\sum_{u=0}^{2h-2} \frac{(4\pi m)^u}{u!}\right),
\end{equation}
or equivalently, employing \e{euler} and the incomplete $\G$ function,
\begin{equation}\label{lerchxxx2}
  \zeta(2h-1) = \frac{(4\pi)^{2h-1}|B_{2h}|}{2(2h)!} -\sum_{m=1}^\infty \sigma_{1-2h}(m) \Big(e^{-2\pi m}+ \G(2h-1,4\pi m) \cdot e^{2\pi m} \Big).
\end{equation}

As we see in Section \ref{harmo}, $V_k(z)$ arises along with $U_k(z)$ in a naturally occurring harmonic Maass form of holomorphic weight $k$ that seems to have been first studied by Pribitkin in \cite{Pr00}. 
Define $\varepsilon(k)$ to be $2$ if $k\gqs 0$ and $1$ if $k< 0$.
Then for  all $k \in 2\Z$, the Maass form in question, by an application of  Theorem \ref{main2}, is
\begin{equation*}
 \ei_k(z) =
1  +
\frac{\varepsilon(k)}{\zeta(1-k)} \left[\frac{\zeta(2-k)}{(2\pi i)^k} \left(\frac{y}{\pi}\right)^{1-k}
 +
U_k(z)
+
   V_k(z)\right].
\end{equation*}


\section{Applying the raising and lowering operators} \label{maa}
Define the raising and lowering operators of Maass as
$$
R_k := 2iy \frac{\partial}{\partial z} + \frac k2 \qquad \text{and} \qquad L_k := -2iy \frac{\partial }{\partial \overline{z}} - \frac k2
$$
respectively. Here $\frac{\partial }{\partial z} := \frac 12(\frac \partial {\partial x}-i\frac \partial {\partial y})$, $\frac{\partial }{\partial \overline{z}} := \frac 12(\frac \partial {\partial x}+i\frac \partial {\partial y})$
and it is easy to check that $L_k f = \overline{R_{-k} \overline{f}}$.
If a function has non-holomorphic weight $k$, in the sense of \e{wtk}, then applying $R_k$ raises its weight to $k+2$ and applying $L_k$ lowers its weight to $k-2$; see for example \cite[Lemma 2.1.1]{Bu}. The weight $k$ Laplacian is
\begin{equation*}
  \Delta_k :=-4y^2\frac{\partial ^2}{\partial z \partial \overline{z}}+ik y \left(\frac{\partial }{\partial z} + \frac{\partial }{\partial \overline{z}}\right)
= -y^2 \left( \frac{\partial ^2}{\partial x^2} +  \frac{\partial ^2}{\partial y^2}\right) +ik y \frac{\partial }{\partial x}
\end{equation*}
and, as in \cite[Sect. 2.1]{Bu}, satisfies
\begin{equation} \label{delsat}
  \Delta_k = -L_{k+2} \circ R_k -\frac k2\left(1+\frac k2 \right) = -R_{k-2} \circ L_k +\frac k2\left(1-\frac k2 \right).
\end{equation}
Our definitions of $R_k,$ $L_k$ and $\Delta_k$ follow those of Bump in \cite{Bu} and have a symmetrical effect on the Eisenstein series as we see next. Maass's original operators were $K_k=R_k$, $\Lambda_k=-L_k$ and he used the Laplacian with the opposite sign. These conventions of Maass are followed in \cite{Ja94} and \cite[Sect. 4]{DFI02}, for example. To act on spaces of functions transforming with holomorphic weight $k$, such as harmonic Maass forms, the operators must be adjusted; see for example \e{egex} at the start of Section \ref{harmo}.


Assuming $k \in 2\Z$ and $\Re(s)>1$, a calculation with \e{defe} shows
\begin{align}
R_k E_k(z,s) &= (s+k/2)E_{k+2}(z,s), \label{reis}\\
  L_k E_k(z,s) &= (s-k/2)E_{k-2}(z,s). \label{leis}
\end{align}
Hence
\begin{align}
R_k E^*_k(z,s) &= \begin{cases} E^*_{k+2}(z,s) \quad & k \gqs 0; \\
 (s+|k|/2-1)(s-|k|/2) E^*_{k+2}(z,s) \quad & k < 0, \end{cases} \label{raise}\\
  L_k E^*_k(z,s) &= \begin{cases} E^*_{k-2}(z,s) \quad & k \lqs 0; \\
 (s+|k|/2-1)(s-|k|/2) E^*_{k-2}(z,s) \quad & k > 0. \end{cases} \label{lower}
\end{align}
As a consequence of \e{delsat} -- \e{lower}, for $k \in 2\Z$ and $\Re(s)>1$,
\begin{equation}\label{eigeis}
  \Delta_k E_k(z,s) = s(1-s) E_k(z,s), \qquad \Delta_k E^*_k(z,s) = s(1-s) E^*_k(z,s).
\end{equation}

A convenient notation for $n$ applications of the raising operator, going from weight $0$ to weight $2n$, is
$$
  R^n_0  := R_{2n-2} \circ R_{2n-4}  \circ \cdots  \circ R_{2}  \circ R_{0}.
$$
Therefore, with $k/2 \in \Z_{\gqs 0}$ and $\Re(s)>1$,
\begin{equation}\label{multr}
  R_0^{k/2}E^*_0(z,s) = E^*_k(z,s)
\end{equation}
and to prove Theorem \ref{allk2} we apply $R_0^{k/2}$ to \e{e0}.
The next lemma is required for this and was first derived in \cite[Sect. 5]{O02} by simplifying recurrences. We give a new proof based on the properties of the polynomials $P_r^n(x)$ defined in \e{pnr}. For $n \gqs 0$ these polynomials may also be expressed as
\begin{equation*}
  P_r^n(x) = \binom{2n}{n+r} \sum_{\ell=|r|}^n  (n-\ell)! \binom{n+r}{\ell+r} \binom{n-r}{\ell-r}  (-x)^\ell
\end{equation*}
or, in terms of the generalized Laguerre polynomials $L_m^{(\alpha)}(x)$, as
\begin{equation} \label{lag}
  P_r^n(x) = \frac{(2n)!}{(n+r)!}(-x)^r L_{n-r}^{(2r)}(x) \qquad \text{with} \qquad L_m^{(\alpha)}(x):=\sum_{j=0}^m \binom{m+\alpha}{m-j}\frac{(-x)^j}{j!}.
\end{equation}
\begin{lemma} \label{rkw}
For all $k/2 \in \Z_{\gqs 0}$, $m\in \R_{\neq 0}$ and $z\in \H$ we have
\begin{equation}\label{twai}
R^{k/2}_0 W_s(mz) = 2^{-k}\sum_{r=-k/2}^{k/2} \left(\frac{m}{|m|}\right)^r  P^{k/2}_r(4\pi my) \cdot W_{s+r}(mz).
\end{equation}
\end{lemma}
\begin{proof}
Verify that
\begin{equation*}
  4\frac{d}{dx}P_r^n(x) = \frac 1x P_r^{n+1}(x)+\left(2-\frac{4n+2}{x} \right)P_r^n(x) + P_{r+1}^n(x)
+ P_{r-1}^n(x)
\end{equation*}
and hence
\begin{multline} \label{oo1}
  y \frac{d}{dy} P_r^{k/2}(4\pi my) = \frac 14 P_r^{k/2+1}(4\pi my) +\left(2\pi m y-\frac{k+1}2\right) P_r^{k/2}(4\pi my)
\\
+\pi m y \left(P_{r+1}^{k/2}(4\pi my) +P_{r-1}^{k/2}(4\pi my)\right).
\end{multline}
We also require the identity
\begin{equation}\label{oo2}
2iy \frac{d}{dz} W_{s+r}(m z) = -\pi |m|y \bigl( W_{s+r-1}(m z) + W_{s+r+1}(m z) \bigr) +\left(\frac 12-2\pi m y\right) W_{s+r}(m z) 
\end{equation}
which follows from
$
  2\frac{d}{dy} K_w(y) = -K_{w+1}(y)- K_{w-1}(y)
$.
The proof now proceeds by induction on $k/2$. The equality \e{twai} is true when $k=0$ and if it holds for $k$ then
\begin{align*}
  R^{(k+2)/2}_0 W_s(mz) & = R_k \circ R^{k/2}_0 W_s(mz) \\
   & = 2iy \frac{d}{dz}\left(R^{k/2}_0 W_s(mz)\right) +\frac k2 \left(R^{k/2}_0 W_s(mz)\right),
\end{align*}
giving
\begin{multline*}
  R^{(k+2)/2}_0 W_s(mz) = 2^{-k}\sum_{r=-k/2}^{k/2} \left(\frac{m}{|m|}\right)^r  \bigg\{ y \frac{d}{dy} P_r^{k/2}(4\pi my) \cdot W_{s+r}(mz) \\
  +P_r^{k/2}(4\pi my) \cdot 2iy \frac{d}{dz} W_{s+r}(mz) + \frac k2 P_r^{k/2}(4\pi my) \cdot W_{s+r}(mz) \bigg\}.
\end{multline*}
Simplifying this with \e{oo1} and \e{oo2} shows the induction step and completes the proof.
\end{proof}

For all $s\in \C$ an induction shows
\begin{equation}\label{alz}
 R^{k/2}_0 \big(\theta(s) y^s \big) = \theta_k(s) y^s
\end{equation}
and it is easy to check that $\theta_k(s)y^s + \theta_k(1-s)y^{1-s}$ for $k \neq 0$ is an entire function of $s$.  Lemma \ref{rkw} and \e{alz} show that formally applying the raising operator to each term in the weight $0$ expansion \e{e0} produces the weight $k\gqs 0$ expansion \e{ekzs2b}.
The relation \e{connj} allows us to access negative weights and is equivalent to applying the lowering operator. For this, conjugate the right side of \e{twai} and replace $s$ with $\overline{s}$.
With (\ref{whit}) it  may be verified that
$
\overline{W_{\overline{s}+r}(m z)} = W_{s+r}((-m)z)
$
and hence, with our definition of $P^n_r(x)$ for negative $n$ after \e{pnr}, we formally obtain  \e{ekzs2b} for negative $k$ also. In the next section we make the necessary estimates to prove these expansions are valid.

\section{Bounds for $E^*_k(z,s)$} \label{bounds}

\subsection{Initial estimates}
Note that \e{bessel} implies $|K_w(y)|\lqs K_r(y)$ for all $w\in \C$ with real part $r$. Also $K_r(y)$ is always positive and equals $K_{-r}(y)$.

\begin{lemma} \label{kbss}
For $r\in \R$ and $y>0$,
\begin{equation} \label{thd}
  K_r(y) <2^{2|r|+1}\left(1+\frac{\G(|r|+1)}{y^{|r|+1}} \right) e^{-y}.
\end{equation}
\end{lemma}
\begin{proof}
Assume that $r \gqs 0$. We have
\begin{align*}
  2K_r(y) = K_r(y)+K_{-r}(y) & = \frac 12 \int_0^\infty e^{-y(t+1/t)/2} (t^{r}+ t^{-r})\, \frac{dt}{t} \\
 & =  \int_1^\infty e^{-y(t+1/t)/2} (t^{r}+ t^{-r})\, \frac{dt}{t}.
\end{align*}
For $T\gqs 1$ the lower part of this last integral is
\begin{equation} \label{bunr}
  \int_{1}^T e^{-y(t+1/t)/2} (t^{r}+ t^{-r})\, \frac{dt}{t} < \int_{1}^T e^{-y} (2 T^{r})\, \frac{dt}{t} = 2e^{-y} T^r \log T.
\end{equation}
The upper part is
\begin{align}
  \int_{T}^\infty e^{-y(t+1/t)/2} (t^{r}+ t^{-r})\, \frac{dt}{t} & <
  \int_{T}^\infty e^{-y t/2} (2 t^{r-1})\, dt \notag\\
   & = 2 \int_{0}^\infty e^{-y (T+u)/2} (T+u)^{r}\, \frac{du}{T+u} \label{icew}\\
& < \frac 2T e^{-y T/2} \int_{0}^\infty e^{-y u/2} (T+u)^{r}\, du. \notag
\end{align}
Using the bound $(T+u)^{r} \lqs (2T)^r+(2u)^r$ for $T,u,r\gqs 0$, computing the resulting integrals and adding to \e{bunr} shows
\begin{equation*}
  K_r(y)  < e^{-y} T^r \log T +  \frac{2}{T} e^{-y T/2} \left(\frac{(2T)^{r}}y+ \frac {4^r \G(r+1)}{y^{r+1}} \right).
\end{equation*}
 Choosing $T=2$ and simplifying completes the proof.
\end{proof}
If in \e{icew} we instead use
\begin{equation*}
  \frac{(T+u)^r}{T+u} \lqs \frac{(2T)^r+ (2u)^r}{T+u} \lqs 2^r\left(T^{r-1}+u^{r-1}\right),
\end{equation*}
then the same arguments lead to
\begin{equation} \label{thd2}
  K_r(y) <4^{|r|}\left(1+\frac 1y+\frac{\G(|r|)}{y^{|r|}} \right) e^{-y} \qquad \text{for} \qquad y>0,r\in \R_{\neq 0}.
\end{equation}
This improves \e{thd} for small $y$ but is not valid at $r=0$ as the $\G$ function has a pole there.

\begin{prop} \label{trkr} Let $S_k(z,s)$ be the series over $m$ on the right of \e{ekzs2b}. For all $z\in \H$ and $s\in \C$ this series  is absolutely convergent and satisfies $S_k(z,s)=S_k(z,1-s)$. Bounding the absolute value of its terms for $\sigma:=\Re(s) \gqs 1/2$ yields
\begin{equation} \label{skiz}
  S_k(z,s) \ll \left( y^{|k|/2+1/2}+y^{-\sigma -|k|/2-1}\right) e^{-2\pi y}
\end{equation}
for an implied constant depending only on $k$ and $s$.
\end{prop}
\begin{proof}
As we saw in \e{fek}, the terms of $S_k(z,s)$ are invariant as $s \to 1-s$, so we may assume $\sigma \gqs 1/2$.
Simple bounds show
\begin{equation*}
  \left| |m|^{-s} \sigma_{2s-1}(|m|)\right| \lqs |m|^\sigma, \qquad \left| P_r^{k/2}(4\pi my) \right| \lqs (4\pi)^{|k|/2}(|k|+1)!\left(1+(my)^{|k|/2} \right)
\end{equation*}
and so we may also assume for simplicity that $k\gqs 0$.
Hence
\begin{equation} \label{oih}
  S_k(z,s) \ll y^{1/2}\sum_{m=1}^\infty m^{\sigma +1/2}\left(1+(my)^{k/2}\right)\sum_{r=-k/2}^{k/2} K_{\sigma+r-1/2}(2\pi m y)
\end{equation}
and Lemma \ref{kbss} implies
\begin{equation} \label{oih2}
  S_k(z,s)  \ll y^{1/2}\sum_t \sum_{m=1}^\infty m^{\sigma +1/2}\left(1+(my)^{k/2}\right)\left(1+\frac{1}{(m y)^{t}}\right) e^{-2\pi m y}
\end{equation}
where $t\gqs 1$ takes the values $|\sigma+r-1/2|+1$ for integers $r$ with $-k/2\lqs r\lqs k/2$.

For $y>0$, $r\in \R$ and $r':=\max(0,r)$, the inequality
\begin{equation}\label{bbgd}
\sum_{m=1}^\infty m^r e^{-m y} \lqs e^{-y} \left(1+2^{r'} e^{-y} + 3^{r'} \frac{\G(r'+1)}{y^{r'+1}} \right)
\end{equation}
 follows by comparing the series on the left to the $\G$ function integral.   Expanding \e{oih2} and employing \e{bbgd} proves \e{skiz}. The $y^{-\sigma -|k|/2-1}$ term in \e{skiz} comes from estimating
\begin{equation*}
  y^{1/2} \sum_{m=1}^\infty m^{\sigma +1/2}\frac{1}{(m y)^{t}} e^{-2\pi m y} \qquad \text{when} \qquad t=\sigma+k+1/2. \qedhere
\end{equation*}
\end{proof}

\subsection{The analytic continuation of $E_k^*(z,s)$}

\begin{proof}[Proof of Theorem \ref{allk2}]
Let
\begin{equation*}
  S_k^*(z,s) :=  \theta_k(s) y^s+\theta_k(1-s) y^{1-s} +S_k(z,s)
\end{equation*}
 be the right side of \e{ekzs2b}. For $\Re(s)>1$ and $k=0$ we have seen in \e{e0} that $E_0^*(z,s) = S_0^*(z,s)$.  Applying $R_0$ to $E_0^*(z,s)$ gives $E_2^*(z,s)$ and, continuing this procedure as  in \e{multr}, $R_0^{k/2}E^*_0(z,s)$ equals $E_{k}^*(z,s)$.
The terms of $S_k^*(z,s)$ are differentiable in $x$ and $y$ with continuous derivatives. We saw in Section \ref{maa} that $S_k^*(z,s)$ is obtained by repeatedly raising the terms of $S_0^*(z,s)$ and is absolutely convergent by Proposition \ref{trkr}. Suppose we have established that $E_k^*(z,s) = S_k^*(z,s)$. To show that $E_{k+2}^*(z,s) = S_{k+2}^*(z,s)$, we require the uniform convergence of the raising operator's derivatives. Let $S_k(z,s)_x$ be the series $S_k(z,s)$ but with each term replaced by its partial derivative with respect to $x$ and define similarly  $S_k(z,s)_y$ with the partial derivatives with respect to $y$. Then $S_k(z,s)_x$ may be bounded as in \e{oih} but with an extra $m$ factor. The proof of Proposition \ref{trkr} goes through and shows that $S_k(z,s)_x$ converges uniformly for $z$ in compact subsets of $\H$. The terms of $S_k(z,s)_y$ may be computed with \e{oo1} and \e{oo2}. We find they are bounded by the estimates for  $S_{k+2}(z,s)$ in Proposition \ref{trkr} and hence also converge uniformly.

These results prove that \e{ekzs2b} is true for $\Re(s)>1$ and even $k\gqs 0$. The relation \e{connj} extends this to $k<0$. The estimates in Proposition \ref{trkr} are valid for all $s \in \C$. Since the bounds from Lemma \ref{kbss} and \e{bbgd} are uniform for $s$ in compact subsets of $\C$, Proposition \ref{trkr} also shows that, for each $z \in \H$, $S_k(z,s)$ is an entire function of $s$. In this way the Fourier expansion $S_k^*(z,s)$ gives the analytic continuation of $E_k^*(z,s)$ to all $s \in \C$. The only poles come from the constant term when $k=0$.
\end{proof}

The next corollary now follows  from Theorem \ref{allk2} and Proposition \ref{trkr}; see \cite[Cor. 3.5]{IwSp} for a similar estimate when $k=0$. We will need this result, and in particular its bound on $E^*_k(z,s)$  as $y \to 0$,  in the proof of Proposition \ref{RankinS}.

\begin{cor} \label{bnthc}
For all  $s\in \C$ with real part $\sigma$, we have
\begin{equation*}
 E^*_k(z,s)= \theta_k(s) y^s+\theta_k(1-s) y^{1-s} +O\left( \left(y^{|k|/2+1/2}+y^{-|\sigma-1/2|-|k|/2-3/2}\right)e^{-2\pi y}\right)
\end{equation*}
where the implied constant depends only on $k \in 2\Z$ and $\sigma$.
\end{cor}

By analytic continuation, the fundamental non-holomorphic weight $k$ transformation property \e{wtk} remains valid for $E^*_{k}(z,s)$ for all $s\in \C$. We next prove that the eigenvalue identity \e{eigeis} also extends to all $s$ (avoiding the poles $s=0,1$ when $k=0$).
With Theorem \ref{allk2} we have shown that for even $k \gqs 0$
\begin{equation}\label{e0kk}
E^*_{k}(z,s) = R_0^{k/2}\left( \theta(s) y^s + \theta(1-s) y^{1-s}\right) + \sum_{m \in \Z_{\neq 0}} \frac{ \sigma_{2s-1}(|m|)}{|m|^s} R_0^{k/2}\left(W_s(m z)\right).
\end{equation}
The relations in \e{delsat} imply $\Delta_k \circ  R_{k-2}=R_{k-2} \circ \Delta_{k-2}$ so that, for all $s\in \C$,
\begin{align*}
 \Delta_k E^*_{k}(z,s) & = \Delta_k \circ R_0^{k/2}\left( \theta(s) y^s+ \theta(1-s) y^{1-s}\right) + \sum_{m \in \Z_{\neq 0}} \frac{ \sigma_{2s-1}(|m|)}{|m|^s}  \Delta_k \circ  R_0^{k/2}\left(W_s(m z)\right)\\
   & =  R_0^{k/2} \circ \Delta_0\left( \theta(s) y^s+ \theta(1-s) y^{1-s}\right) + \sum_{m \in \Z_{\neq 0}} \frac{ \sigma_{2s-1}(|m|)}{|m|^s}    R_0^{k/2} \circ \Delta_0\left(W_s(m z)\right).
\end{align*}
Applying the identities
\begin{equation} \label{utah}
  \Delta_0 y^s = s(1-s) y^s, \qquad \Delta_0 W_s(m z) = s(1-s) W_s(m z)
\end{equation}
now gives
\begin{equation}\label{alls}
  \Delta_k E^*_{k}(z,s) = s(1-s) E^*_{k}(z,s) \qquad \text{for all} \qquad s\in \C
\end{equation}
where, as usual, the relation for $k<0$ follows from \e{connj}. Similar methods  extend the identities \e{reis} -- \e{lower} to all $s \in \C$.

\subsection{Whittaker functions}

Following the original derivation of Maass, the Fourier development of $E^*_k(z,s)$ is often given in terms of the
Whittaker functions $W_{\kappa,\mu}$.
In our normalization, Maass's result \cite[p. 210]{Maa} is
\begin{multline} \label{prb2}
   E^*_k(z,s)= \theta_k(s) y^s+\theta_k(1-s) y^{1-s}
+ \frac{(-1)^{k/2}}{\G(s-|k|/2)}\sum_{m=1}^\infty \frac{\sigma_{2s-1}(m)}{m^{s}}\\
\times \Bigl\{\G(s-k/2) \cdot W_{k/2,s-1/2}(4\pi m y)  \cdot e^{2\pi i m x}
+ \G(s+k/2) \cdot W_{-k/2,s-1/2}(4\pi m y)  \cdot e^{-2\pi i m x} \Bigr\}.
\end{multline}
For this see  \cite[Sect. 2]{ALR} where the results of Maass are translated into the $E_k(z,s)$ form.
The development \e{prb2} is equivalent to Eqs. (1.6), (1.7) in \cite{Ja94} when $s=1/2+it$.

The Whittaker functions satisfy
\begin{equation} \label{whde}
  \frac{d^2}{dy^2} W_{\kappa,\mu}(y)+\left[-\frac{1}{4}+\frac{\kappa}{y}+\frac{1/4-\mu^2}{y^2} \right] W_{\kappa,\mu}(y) =0,
\end{equation}
have exponential decay as $y\to \infty$ and possess various integral representations as in \cite[Chap. 16]{WW} and \cite[Sect. 13.16]{DLMF}. They are entire in the $\kappa$ and $\mu$ parameters with $W_{\kappa,\mu}(y) = W_{\kappa,-\mu}(y)$.
The Eisenstein series are eigenfunctions of $\Delta_k$ as we saw in \e{alls} and of period $1$ in $x$. Writing their Fourier expansion and separating variables gives a differential equation for the coefficients that may be transformed into \e{whde}.
This shows that  the terms in \e{prb2} must take the form they do, and for $n \in \Z_{\neq 0}$
\begin{equation}\label{ggw}
  \left( \Delta_k -s(1-s)\right)\left( W_{k\cdot n/|n|,s-1/2}(4\pi |n| y) e^{2\pi i nx}\right) = 0.
\end{equation}
The relationship of $W_s(z)$ from \e{whit} with  $W_{\kappa,\mu}(y)$ is
$$
W_s(z)=W_{0,s-1/2}(4\pi |y|)\cdot e^{2\pi i x}
$$
and the right equation in \e{utah} is the $k=0$ case of \e{ggw}.

The Fourier expansion of weight $k$ Eisenstein series for the group $\G_0(N)$ is described in \cite[Sect. 3]{AD}.
Similar developments to  \e{prb2} but employing other confluent hypergeometric variants are \cite[Eq. (1)]{Pr00}, \cite[Cor. 2.2]{KN09} and \cite[Thm. 7.2.9]{Miy} where characters are included.

Comparing \e{prb2} with Theorem \ref{allk2} provides the Whittaker function relation:
\begin{cor}
For all $k\in 2\Z$, $s\in \C$ and $y>0$,
\begin{equation} \label{whyt}
  W_{k/2,s-1/2}(y)=\frac{(-1)^{k/2}}{2^{|k|}}\frac{\G(s-|k|/2)}{\G(s-k/2)}\sum_{r=-|k|/2}^{|k|/2}P^{k/2}_r(y)\cdot W_{0,s+r-1/2}(y).
\end{equation}
\end{cor}

\section{The Eisenstein series $E^*_k(z,s)$ at integer values of $s$} \label{integer}

In this section we prove Theorem \ref{main2}. The Bessel function $K_w(y)$ may be expressed as a rational function times the exponential function when $w \in \Z+1/2$ and the simple idea of the proof is to put this expression into Theorem \ref{allk2}.

\subsection{Raising $W_h(mz)$}
For the falling factorial, write $z^{\underline{n}} := z(z-1) \cdots (z-n+1)$
with $z \in \C$ and $n \in \Z_{\gqs 0}$ (and $z^{\underline{0}}=1$).  The generalized binomial coefficients are given by
\begin{equation}\label{bin}
\binom{z}{n}:=\begin{cases} z^{\underline{n}}/n! & \text{ \ if \ } n\in \Z_{\geqslant 0},\\
0  & \text{ \ if \ } n\in \Z_{< 0}.\end{cases}
\end{equation}
Recalling \e{hstar}, we see that
for $n,j \in \Z$
\begin{equation} \label{binb}
  \binom{n-1+j}{2j} \neq 0 \quad \text{if and only if} \quad 0 \lqs j \lqs n^*.
\end{equation}

\begin{lemma} \label{kb}
For all $n \in \Z$ and $y>0$,
\begin{equation}
   K_{n-1/2}(y)=\left(\frac{\pi}{2y}\right)^{1/2} e^{-y} \sum_{j= 0}^{n^*} \binom{n-1+j}{2j} \frac{(2j)!}{j! (2y)^j} . \label{kbess}
\end{equation}
\end{lemma}
\begin{proof}
The formula \e{kbess} is well-known for $n=1$. Use $K_{w+1}(y)=\frac wy K_w(y)-\frac{d}{dy} K_w(y)$ and induction to verify (\ref{kbess}) for all larger integer values of $n$. The identity $K_w(y)=K_{-w}(y)$ shows that the left side of \e{kbess}
does not change as $n \to 1-n$. It is elementary to check that the right side is also unchanged as $n \to 1-n$ and so \e{kbess} is  true for all $n \in \Z_{\lqs 0}$ as well.
\end{proof}

Recall the definition of $\mathcal A^k_h(u)$ in \e{defak2}.

\begin{prop} \label{long} For $k/2 \in \Z_{\gqs 0}$, $h \in \Z$ and $m\in \R_{\neq 0}$ we have
\begin{equation*}
  R^{k/2}_0 W_h(m z) = \begin{cases}
 \displaystyle e^{2\pi i m z} \sum_{u=0}^{h^* +k/2} \mathcal A^k_h(u) \cdot (4\pi |m|y)^{-u+k/2}
 & \text{ \ if \ } m> 0,\\
 \displaystyle  e^{2\pi i m \overline{z}} \sum_{u=0}^{h^*-k/2^{\phantom{\big |}}} \mathcal A^{-k}_h(u) \cdot (4\pi |m|y)^{-u-k/2}  & \text{ \ if \ } m < 0.
\end{cases}
\end{equation*}

\end{prop}
\begin{proof}
Put $\delta := m/|m|$. By Lemmas \ref{rkw} and \ref{kb} we have
\begin{multline*}
  R^{k/2}_0 W_h(m z) = 2^{-k} \sum_{|r| \leqslant k/2} \delta^r P^{k/2}_r(4 \pi m y) W_{h+r}(mz) \\
    = \frac{e^{2\pi i m x- 2\pi |m|y}}{2^k} \sum_{|r| \leqslant k/2} \delta^r \sum_{\ell=|r|}^{k/2}\frac{k! (-4 \pi my)^\ell}{(k/2-\ell)! (\ell+r)! (\ell-r)!}
   \sum_{j=0}^{(h+r)^*} \binom{h+r-1+j}{2j}\frac{(2j)!}{j! (4\pi |m| y)^j}.
\end{multline*}
Interchanging the order of summation and writing $v=j-\ell$ gives
\begin{equation}\label{ab}
    R^{k/2}_0 W_h(m z)= \frac{k!}{2^k} \ e^{2\pi i m x- 2\pi |m|y} \sum_{v=-k/2}^{h^*} (-\delta)^{v} \frac{\alpha_k(h,v;\delta)}{(4 \pi |m|y)^v}
\end{equation}
for
\begin{align}
  \alpha_k(h,v;\delta) & := \sum_{j=\max(0,v)}^{k/2+v} (-\delta)^{j} \frac{(2j)!}{j! (k/2+v-j)!} \beta(h,v,j;\delta), \label{the}\\
  \beta(h,v,j;\delta) & := \sum_{|r| \leqslant j-v}  \binom{h+r-1+j}{2j} \frac{\delta^r}{(j-v+r)!(j-v-r)!}. \label{ps}
\end{align}
The sum in \e{ps} is initially over all $r$ such that $|r| \leqslant j-v$ and $j \leqslant (h+r)^*$. With \e{binb} the condition $j \leqslant (h+r)^*$ may be removed.

To simplify these formulas we first assemble the combinatorial results we shall need.

\begin{lemma} For all $a,b,c \in \Z$ with $a \geqslant 0$, we have
\begin{align}
  \sum_{\ell=0}^a \binom{a}{\ell} \binom{b+\ell}{c} &= \sum_{i=0}^a \binom{a}{i} \binom{b}{c-i} 2^{a-i}, \label{comb1}\\
  \sum_{\ell=0}^a  \binom{a}{\ell} \binom{b}{c-\ell} &= \binom{a+b}{c}, \label{comb4}\\
   \sum_{\ell=0}^a (-1)^\ell \binom{a}{\ell} \binom{b+\ell}{c} &= (-1)^a \binom{b}{c-a}, \label{comb2}\\
   \sum_{\ell=0}^a (-1)^\ell \binom{a}{\ell} \binom{2\ell}{c} &= (-1)^a \binom{a}{c-a}2^{2a-c}. \label{comb3}
\end{align}
\end{lemma}
\begin{proof}
 Recall that the general binomial theorem implies
$$
(1+x)^n=\sum_{\ell=0}^\infty \binom{n}{\ell} x^\ell
$$
for all $x$ when $n\in \Z_{\gqs 0}$ and for all $x$ with $|x|<1$ when $n \in \Z_{<0}$. To prove (\ref{comb1}) we evaluate in two ways $(1+1+x)^a(1+x)^b$.
For $|x|<1$,
\begin{equation}\label{l1}
(1+(1+x))^a(1+x)^b = \sum_{\ell=0}^a \binom{a}{\ell} (1+x)^{b+\ell} = \sum_{\ell=0}^a\sum_{c=0}^\infty \binom{a}{\ell}\binom{b+\ell}{c}x^c
\end{equation}
and also
\begin{align}
(2+x)^a(1+x)^b & = 2^a (1+x/2)^a(1+x)^b \notag \\
& = 2^a \sum_{i=0}^a \sum_{j=0}^\infty \binom{a}{i}\binom{b}{j} (x/2)^i x^j = \sum_{i=0}^a \sum_{c=0}^\infty \binom{a}{i}\binom{b}{c-i} 2^{a-i} x^c. \label{l2}
\end{align}
Comparing coefficients of $x^c$ in (\ref{l1}), (\ref{l2}) yields (\ref{comb1}). Similar proofs using the elementary identities
\begin{align*}
  (1+x)^a(1+x)^b & =  (1+x)^{a+b}, \\
  (1-(1+x))^a(1+x)^b & =  (-x)^a (1+x)^b, \\
  (1-(1+x)^2)^a & =  (-x)^a(2+x)^a
\end{align*}
 give (\ref{comb4}), (\ref{comb2}) and (\ref{comb3}) respectively.
\end{proof}

In the following two lemmas the integers $k,$ $h,$ $v$ and $j$ are restricted to the ranges required in \e{ab} and \e{the}. Explicitly: $k/2 \in \Z_{\gqs 0}$,  $-k/2 \lqs v \lqs h^*$ and $\max(0,v) \lqs j \lqs k/2+v$.
\begin{lemma} We have
\begin{align}
  \beta(h,v,j;-1) &= \frac{(-1)^{j+v} }{(2j-2v)!} \binom{h+v-1}{2v}, \label{psi-1}\\
  \beta(h,v,j;1) &= \sum_{\ell=0}^{2j-2v}  \frac{2^\ell }{\ell! (2j-2v-\ell)!} \binom{h+v-1}{2v+\ell}. \label{psi1}
\end{align}
\end{lemma}
\begin{proof}
Put $t=r+j-v$ to get
$$
\beta(h,v,j;\delta)= \frac{\delta^{j+v} }{(2j-2v)!} \sum_{t=0}^{2j-2v} \delta^t \binom{2j-2v}{t} \binom{h+v-1+t}{2j}.
$$
Apply (\ref{comb1}) and (\ref{comb2}) to complete the proof.
\end{proof}
\begin{lemma} \label{thetaa}
We have $\alpha_k(h,v;-1)=0$ unless $k/2\lqs v \lqs h^*$. Also
\begin{align}\label{alx}
  \alpha_k(h,v;-1) & = (-1)^{k/2} \frac{2^k}{k!} \frac{(h+v-1)^{\underline{2v}}}{(v-k/2)!} \qquad \text{for} \qquad k/2\lqs v \lqs h^*, \\
 \alpha_k(h,v;1) & = (-1)^{k/2+v} \frac{2^k}{k!} (h+v-1)^{\underline{v+k/2}} \binom{h+k/2-1}{v+k/2}. \label{alx2}
\end{align}
\end{lemma}
\begin{proof}
Put \e{psi-1} into \e{the} and use \e{binb}, \e{comb3} to simplify and obtain the stated formulas for $\alpha_k(h,v;-1)$.
Put (\ref{psi1}) into (\ref{the}) to get
\begin{align*}
  \alpha_k(h,v;1) &=  \frac{1}{(k/2+v)!} \sum_{j=0}^{k/2+v} (-1)^j \binom{k/2+v}{j} \sum_{\ell=0}^{2j-2v}  \frac{2^\ell (2j)!}{\ell! (2j-2v-\ell)!} \binom{h+v-1}{2v+\ell}\\
   &=  \frac{1}{(k/2+v)!} \sum_{\ell=0}^{k} \frac{2^{\ell}}{\ell!} \binom{h+v-1}{2v+\ell} (2v+\ell)!
   \sum_{j=v+\lceil \ell/2 \rceil}^{v+k/2} (-1)^j \binom{k/2+v}{j}   \binom{2j}{2v+\ell}.
\end{align*}
Using (\ref{comb3}) to find the inner sum produces
\begin{equation}\label{the2}
\alpha_k(h,v;1) = (-1)^{k/2+v} 2^k \sum_{\ell=0}^{k} \frac{(h+v-1)^{\underline{2v+\ell}}}{(v-k/2+\ell)! \ (k-\ell)! \ \ell!}.
\end{equation}
Breaking up the numerator as
$$
(h+v-1)^{\underline{2v+\ell}} = (h+v-1)^{\underline{v+k/2}} \times (h-k/2-1)^{\underline{v-k/2+\ell}},
$$
shows that \e{the2} may be rewritten as
\begin{equation*}
\alpha_k(h,v;1)  =   (-1)^{k/2+v} \frac{2^k}{k!}(h+v-1)^{\underline{v+k/2}} \sum_{\ell} \binom{h-k/2-1}{v+k/2-\ell} \binom{k}{\ell}
\end{equation*}
and an application of \e{comb4} is the last step.
\end{proof}

We may now complete the proof of Proposition \ref{long} by
using Lemma \ref{thetaa} in \e{ab}. For $m>0$ we obtain
\begin{equation*}
  R^{k/2}_0 W_h(m z)=  e^{2\pi i m z} \sum_{v=-k/2}^{h^*} (-1)^{k/2} (h+v-1)^{\underline{v+k/2}} \binom{h+k/2-1}{v+k/2} (4 \pi |m|y)^{-v}.
\end{equation*}
Letting $u=v-k/2$ this becomes
\begin{equation*}
  R^{k/2}_0 W_h(m z)=  e^{2\pi i m z} \sum_{u=0}^{h^*+k/2} (-1)^{k/2} u! \binom{h-k/2-1+u}{u} \binom{h+k/2-1}{u} (4 \pi |m|y)^{-u+k/2}.
\end{equation*}
Applying the identity
\begin{equation} \label{binneg}
  \binom{-z}{u} = (-1)^u \binom{z+u-1}{u}
\end{equation}
to the first binomial coefficient produces the desired formula.

The case when $m<0$ is similar. By Lemma \ref{thetaa}, the summands in \e{ab} can only be non-zero for $k/2 \lqs v \lqs h^*$ and so we use the new index $u=v+k/2$.
\end{proof}

\subsection{The $\mathcal A$ coefficients}

\begin{prop} \label{aprop} For all $k/2, h \in \Z$ and $u \in \Z_{\gqs 0}$   the following are true.
\begin{enumerate}
\item $\mathcal A^k_h(u) \in \Z.$
\item $\mathcal A^k_{1-h}(u)=\mathcal A^k_h(u)$.
\item We have the representation
\begin{equation}\label{aaak}
\mathcal A^k_h(u) = \frac{(-1)^{k/2}}{u!} \prod^{\max(0,k/2)}_{\ell=k/2+1-u} \left[(h-1/2)^2-(\ell-1/2)^2 \right].
\end{equation}
The product in \e{aaak} is empty when $u=0$ and $k\gqs 0$. In that case the product should be interpreted as $1$ and so $\mathcal A^k_h(0) = (-1)^{k/2}$ for $k\gqs 0$.
\item For  $h^* < k/2$ we have $\mathcal A^k_h(u) \neq 0$ if and only if  \ $0 \leqslant u \leqslant k/2-1- h^*$.
\item For  $h^* \geqslant k/2$ we have $\mathcal A^k_h(u) \neq 0$ if and only if \  $0 \leqslant u \leqslant k/2+ h^*$.
\end{enumerate}
\end{prop}
\begin{proof}
Parts (i) and (ii) follow directly from the definition \e{defak2}.  Also from the definition, for $k \geqslant 0$,
\begin{eqnarray*}
  \mathcal A^k_h(u) &=& \frac{(-1)^{k/2}}{u!} \prod^{u-1}_{r=0} (k/2+h-1-r)(h-k/2+r) \\
  &=& \frac{(-1)^{k/2}}{u!} \prod^{u-1}_{r=0} \left[(h-1/2)-(k/2-r-1/2) \right] \left[(h-1/2)+(k/2-r-1/2) \right] \\
   &=& \frac{(-1)^{k/2}}{u!} \prod^{k/2}_{\ell=k/2+1-u} \left[(h-1/2)^2-(\ell-1/2)^2 \right].
\end{eqnarray*}
Similarly for $k<0$ and part (iii) is verified. To prove parts (iv) and (v), note that (iii) implies that $\mathcal A^k_h(u) =0$ if and only if there exists an $\ell$ in the range $k/2+1-u \leqslant \ell \leqslant \max(0,k/2)$ with $h^* = \ell^*$.

Suppose $h^* < k/2$. This implies that $k>0$. If $k/2+1-u \leqslant h^*+1$ then there exists an $\ell$ with $h^* = \ell^*$, namely $\ell=h^*+1$, and so $\mathcal A^k_h(u) =0$. But if $k/2+1-u > h^*+1$ then $h^* = \ell^*$ is not possible and so $\mathcal A^k_h(u) \neq 0$. This proves part (iv).

Suppose $h^* \gqs k/2$. In this case it is impossible to have $h^* = \ell^*$ for $\ell\gqs 0$. Thus $\mathcal A^k_h(u) =0$ is only possible if we have an $\ell$ equalling $-h^*$. This happens when $k/2+1-u \leqslant -h^*$. Hence $\mathcal A^k_h(u) \neq 0$ if and only if $k/2+1-u > -h^*$, proving part (v).
\end{proof}

As well as \e{defak2} and \e{aaak}, a compact representation of $\mathcal A^k_h(u)$ is given in \cite[(3.1)]{DO10}:
\begin{equation}\label{defak2b}
\mathcal A^k_h(u) = \frac{(-1)^{k/2}}{u!}  \frac{\G(h-k/2+u)}{\G(h-k/2)} \frac{\G(h+|k|/2)}{\G(h+k/2-u)}.
\end{equation}
Note that if any of the $\G$ function arguments above are in $\Z_{\lqs 0}$ then the limiting values of the quotients are meant.

\begin{proof}[Proof of Theorem \ref{main2}]
Using Proposition \ref{long} in \e{e0kk} gives \e{mnk2} for even $k \gqs 0$. Conjugate both sides and use \e{connj} to obtain \e{mnk2} for $k$ negative. With parts (iv) and (v) of Proposition \ref{aprop} we see when the range of the indices may be reduced.
\end{proof}

In \e{whyt} we have already seen a connection between  Whittaker functions and the Laguerre polynomials that we defined  in \e{lag}.
A second relation is the identity
\begin{equation}\label{whi}
  W_{(\alpha+1)/2+n,\alpha/2}(y)=(-1)^n n! \, e^{-y/2}y^{(\alpha+1)/2} L^{(\alpha)}_n(y) \qquad \qquad (n \in \Z_{\gqs 0})
\end{equation}
  from \cite[(13.18.17)]{DLMF}.
Substituting \e{whi} into Maass's formula \e{prb2} provides another route to some cases of Theorem \ref{main2}.
  If \e{whi} is valid for all $\alpha \in \Z$ then it may be verified that we obtain Theorem \ref{main2} in this way when $h^* \gqs |k|/2$, corresponding to Eisenstein series in the left and right triangles in Figure \ref{bfig}.
A referee kindly provided the following proof of \e{whi}, valid for all $\alpha \in \R$ and  $y>0$. From the definition in \cite[Sect. 16.1]{WW},
\begin{equation} \label{wwww}
  W_{(\alpha+1)/2+n,\alpha/2}(y)= - n! \, e^{-y/2} y^{(\alpha+1)/2+n} \frac{1}{2\pi i} \int_{\infty}^{(0+)}
  (-t)^{-n-1}(1+t/y)^{n+\alpha} e^{-t}\,dt
\end{equation}
provided $-y$ is outside the contour of integration. Under the change of variables $t=uy/(1-u)$, the right side of \e{wwww} becomes
\begin{equation}\label{ww99}
   (-1)^n n! \, e^{-y/2} y^{(\alpha+1)/2} \frac{1}{2\pi i} \oint_{|u|=\rho<1}
  u^{-n-1}(1-u)^{-\alpha-1} e^{-uy/(1-u)}\,du.
\end{equation}
The well-known generating function for the Laguerre polynomials is
\begin{equation*}
  \sum_{n=0}^\infty L^{(\alpha)}_n(y) \cdot u^n = (1-u)^{-\alpha-1} e^{-uy/(1-u)}
\end{equation*}
so that \e{ww99} equals the right side of \e{whi} by Cauchy's Theorem.


\section{Values of the spectral zeta function of a torus} \label{torus}
In the following sections we give some applications of Theorems \ref{allk2} and \ref{main2}. This section provides a direct proof of an identity in \cite{CJK10} and we give some of its context here next.

For a $d$-tuple of positive integers $N=(n_1,\dots,n_d)$, the Cayley graph of the product of cyclic groups $\Z/n_1\Z \times \cdots \times \Z/n_d\Z$ may be thought of as a $d$-dimensional discrete torus. The set of eigenvalues of the combinatorial Laplacian acting on this torus is
\begin{equation*}
  \Lambda_N=\Big\{2d-2\cos(2\pi k_1/n_1)- \cdots - 2\cos(2\pi k_d/n_d) \ \Big | \ 0\lqs k_j < n_j, \ 1\lqs j \lqs d \Big\}
\end{equation*}
and the associated spectral zeta function is
\begin{equation*}
 \zeta_N(s):= \sum_{0 \neq \lambda \in \Lambda_N} \frac 1{\lambda^s}.
\end{equation*}
As the components of $N$ go to infinity at comparable rates, it is shown by Chinta,  Jorgenson and Karlsson in \cite{CJK10} that $\zeta_N(s)$ approaches the spectral zeta function of the corresponding real torus; see \cite[Theorem, p. 125]{CJK10}. Combining their asymptotic results with work of Duplantier and David, they obtain in \cite[Eq. (11)]{CJK10} for  the case $d=2$ and $s=2$ the following formula
\begin{multline} \label{cjk}
  \sum_{\substack{m,n \in \Z \\(m,n)\neq (0,0)}} \frac{1}{((my)^2 +n^2)^2} = \left(\frac{2\pi}{y}\right)^2 \left(\frac{(2\pi y)^2}{2^4\cdot 45} + \frac{\zeta(3)}{4\pi y} \phantom{\sum_{n=1}^\infty} \right.\\
  \left. + \frac{1}{2\pi y}\sum_{n=1}^\infty \frac{1}{n^3}\frac{e^{-2\pi n y}}{1-e^{-2\pi n y}}
+ \sum_{n=1}^\infty \frac{1}{n^2}\frac{e^{-2\pi n y}}{(1-e^{-2\pi n y})^2}\right),
\end{multline}
with $y>0$. See \cite[Sections 1.3, 7.3, 7.4]{CJK10} for the details.

We may use Theorem \ref{main2} to give a direct proof of \e{cjk} and also extend it to higher integer values of $s$.
The left side of \e{cjk} may be expressed in
 general, for $\Re(s)>1$, as
\begin{align}
  \sum_{\substack{m,n \in \Z \\(m,n)\neq (0,0)}} \frac{1}{((my)^2 +n^2)^s} & = \sum_{\ell=1}^\infty \sum_{\substack{c,d \in \Z \\ \gcd(c,d)=1}} \frac{1}{(\ell^2 (cy)^2+\ell^2 d^2)^s} \notag\\
   & = \zeta(2s) \sum_{\substack{c,d \in \Z \\ \gcd(c,d)=1}} \frac{1}{|c i y +d|^{2s}} \notag\\
& = \frac{2 \zeta(2s)}{y^s} \sum_{\g \in \G_\ci \backslash \G} \Im(\g(iy))^s = \frac{2}{\G(s)} \left(\frac{\pi}{y}\right)^s E_0^*(iy,s). \label{chin}
\end{align}
With $k=0$ and $h\in \Z-\{0,1\}$, Theorem \ref{main2} gives
\begin{equation} \label{e00}
  E_0^*(z,h) = \theta_0(h) y^h+\theta_0(1-h) y^{1-h} +
   \sum_{m =1}^\infty \frac{ \sigma_{2h-1}(m)}{m^{h}} \left( e^{2\pi i m z}+ e^{-2\pi i m \overline z}\right) \sum_{u=0}^{h^*}  \mathcal A^0_h(u) \cdot (4\pi m y)^{-u}
\end{equation}
for
\begin{equation*}
   \mathcal A^0_h(u) = (-1)^u u! \binom{h-1}{u}\binom{-h}{u} = u!\binom{h-1}{u}\binom{h-1+u}{u}.
\end{equation*}
Also $\theta_0(h)$, $\theta_0(1-h)$ may be found with \e{thet} and so we obtain the weight $k=0$ case of Theorem \ref{main2} when $h\in \Z_{\gqs 2}$:
\begin{multline}\label{cas0}
  E_0^*(z,h) = (h-1)!\zeta(2h)\frac{y^h}{\pi^h} +\frac{(2h-2)!}{(h-1)!}\zeta(2h-1) (4\pi y)^{1-h}  \\
  +
   \sum_{m =1}^\infty \frac{ \sigma_{2h-1}(m)}{m^{h}} \left( e^{2\pi i m z}+ e^{-2\pi i m \overline z}\right) \sum_{u=0}^{h-1}  \frac{(h-1+u)!}{(h-1-u)!u!}  (4\pi m y)^{-u}.
\end{multline}
Note that \e{cas0} agrees with \cite[Eq. (1.13)]{DD}.

To get \e{e00} or \e{cas0} to match \e{cjk}, let  $z=iy$ and substitute $\sigma_{1-2h}(m)/m^{1-h}$ for $ \sigma_{2h-1}(m)/m^{h}$ to find
\begin{equation} \label{omn}
  E_0^*(iy,h) = \theta_0(h) y^h+\theta_0(1-h) y^{1-h} +
   2\sum_{m =1}^\infty \frac{ \sigma_{1-2h}(m)}{m^{1-h}} \ e^{-2\pi  m y} \sum_{u=0}^{h-1}  \mathcal A^0_h(u) \cdot (4\pi m y)^{-u}
\end{equation}
for $h\in \Z_{\gqs 2}$. Reorder
the right of \e{omn} as
\begin{equation} \label{omn2}
  2\sum_{u=0}^{h-1}  \mathcal A^0_h(u) \cdot (4\pi  y)^{-u} \sum_{m =1}^\infty \Bigl( \sum_{d | m} d^{1-2h}\Bigr)  m^{h-1-u} \ e^{-2\pi  m y}
\end{equation}
and letting $m=d \ell$ shows that \e{omn2} becomes
\begin{equation} \label{omn3}
  2\sum_{u=0}^{h-1}  \mathcal A^0_h(u) \cdot (4\pi  y)^{-u} \sum_{d =1}^\infty  d^{-h-u} \sum_{\ell =1}^\infty   \ell^{h-1-u} \ e^{-2\pi  d\ell y}.
\end{equation}
The last series in \e{omn3} may be expressed in terms of the polylogarithm $\pl_{s}(z):=\sum_{n=1}^\infty z^n/n^s$. For indices
 $-m \in \Z_{\lqs 0}$, Euler showed that the polylogarithm is in fact a rational function of the form
\begin{equation} \label{fro}
    \pl_{-m}(z) = \frac{z \cdot A_m(z)}{(1-z)^{m+1}}.
\end{equation}
The first two Eulerian polynomials are constant: $ A_0(z)=1$, $A_1(z)=1$. The next three are
\begin{equation*}
   A_2(z)=1+z,  \quad A_3(z)=1+4z+z^2, \quad A_4(z)=1+11z+11z^2+z^3.
\end{equation*}
Frobenius  showed that
\begin{equation*}
    A_m(z) = \sum_{k=0}^m k! \stirb{m}{k} (z-1)^{m-k} \qquad \quad (m \in \Z_{\gqs 0})
\end{equation*}
with the Stirling number $\stirb{m}{k}$   indicating the number of partitions of  $m$ elements into $k$ non-empty subsets. See for example \cite[Sect. 8.2]{O16}  for more information on these topics.

Assembling our calculations proves the next result, of which \e{cjk} is a special case.

\begin{prop}
For all $h \in \Z_{\gqs 2}$ and $y>0$,
\begin{multline*}
  \sum_{\substack{m,n \in \Z \\(m,n)\neq (0,0)}} \frac{1}{((my)^2 +n^2)^h} =
2\zeta(2h)
+2\pi \zeta(2h-1) \binom{2h-2}{h-1}(4y)^{1-h}\\
+
 \sum_{u=0}^{h-1} \binom{h-1+u}{u} \frac{4^{1-u}\pi^{h-u}}{(h-1-u)!}
\sum_{n=0}^\infty \frac{1}{(ny)^{h+u}} \frac{e^{-2\pi n y}A_{h-1-u}(e^{-2\pi n y})}{(1-e^{-2\pi n y})^{h-u}}.
\end{multline*}
\end{prop}

As an aside we note that \e{e00} may be proved by using Lemma \ref{kb} in \e{e0}. Then another means of proving Theorem \ref{main2} is to apply $R_0^{k/2}$ to \e{e00}. The computations are similar to those of Section \ref{integer} and this was the proof sketched for \cite[Thm. 3.1]{DO10}. The advantage of first proving Theorem \ref{allk2} is that it is a more general result.

\section{The formulations of Maass and Brown} \label{brown}
 Maass defined his Eisenstein series as
\begin{equation} \label{maaasss}
  G(z, \overline{z};\alpha,\beta):= \sum_{\substack{m,n \in \Z \\(m,n)\neq (0,0)}}(mz+n)^{-\alpha}(m\overline{z}+n)^{-\beta}
\end{equation}
for $z\in \H$, $\alpha-\beta \in 2\Z$, $\Re(\alpha+\beta)>2$  and proved fundamental properties such as their Fourier expansion in terms of Whittaker functions, meromorphic continuation in $\alpha,$ $ \beta$, functional equation and eigenvalue properties. This work is described in \cite[Chap. 4]{Maa}. A detailed presentation of these results is also  given in \cite[Sect. 2.2]{ALR}.

For $a,b \in \Z_{\gqs 0}$ with $a+b=w$ positive and even, Brown  in \cite[Sect. 4]{B18} has similarly defined
\begin{equation} \label{brooo}
  \mathcal E_{a,b}(z):=\frac{w!}{(2\pi i)^{w+2}}\frac 12 \sum_{\substack{m,n \in \Z \\(m,n)\neq (0,0)}} \frac{-2\pi y}{(mz+n)^{a+1}(m\overline{z}+n)^{b+1}}.
\end{equation}
With a short calculation similar to \e{chin}, the series \e{maaasss} and \e{brooo} may be related to \e{defe}. This shows
\begin{align}
 G(z, \overline{z};\alpha,\beta) & = \frac{2}{\G((\alpha+\beta)/2+|\alpha-\beta|/2)} \left( \frac y\pi \right)^{(\alpha+\beta)/2}E^*_{\alpha-\beta}(z,(\alpha+\beta)/2), \notag\\
  \mathcal E_{a,b}(z) & =  (-4\pi y)^{-w/2} \frac{\G(w+1)}{2 \G(w/2+1+|a-b|/2)} E^*_{a-b}(z,w/2+1). \label{ebr}
\end{align}
We see that Brown's series $\mathcal E_{a,b}(z)$ corresponds to $E^*_k(z,h)$ with the restriction $|k/2|<h$, and these are the series appearing in the right triangle of Figure \ref{bfig}.
The Fourier expansion of $\mathcal E_{a,b}(z)$  follows as an exercise from Theorem \ref{main2}. With the formulation in Section 4 of \cite{B18} it may be written neatly as
\begin{equation*}
  \mathcal E_{a,b}(z) = \mathcal E^0_{a,b}(z) + R_{a,b}(z) + \overline{R_{b,a}(z)}
\end{equation*}
where
\begin{align*}
  \mathcal E^0_{a,b}(z) & = \frac{\pi y B_{w+2}}{(w+1)(w+2)} + (-1)^a \frac{w!}{2} \binom{w}{a}(4\pi y)^{-w} \zeta(w+1), \\
  R_{a,b}(z) & = \frac{(-1)^a}{2} \sum_{u=0}^a \frac{(a+b)!}{(a-u)!} \binom{b+u}{b} (4\pi y)^{-u-b}
\sum_{m=1}^\infty \frac{\sigma_{w+1}(m)}{m^{b+1+u}}e^{2\pi i m z}.
\end{align*}
This is proved in Sections 5.3 and 5.4 of \cite{B} as part of a larger framework.

\section{A formula of Terras and Grosswald for $\zeta(2h-1)$} \label{terras}

The Chowla-Selberg formula gives a series expansion for the Epstein zeta function $\zeta_Q(s)$ and may be recognized as special case of the expansion \e{e0} of $E_0^*(z,s)$ where $z$ depends on the binary quadratic form $Q$. See  Sections 2 and 3 of \cite{DIT}, for example, for a discussion of these ideas.
Terras generalized the Chowla-Selberg formula to $n$-ary quadratic forms and used this theory in \cite{Te76} to discover a new  formula for $\zeta(2h-1)$, similar to that of Lerch and Ramanujan in \e{z3} and \e{lerch}. The first cases are
\begin{align} \label{z3b}
  \zeta(3) & =\frac{2\pi^3}{45}-2\sum_{m=1}^\infty \sigma_{-3}(m) e^{-2\pi m}\left(1+2\pi m +4\pi^2 m^2\right),\\
\zeta(5) & =\frac{4\pi^5}{945}-2\sum_{m=1}^\infty \sigma_{-5}(m) e^{-2\pi m}\left(1+2\pi m +2\pi^2 m^2 + \frac 43 \pi^3m^3\right),
\end{align}
and in general, for $h \in \Z_{\gqs 2}$,
\begin{multline}\label{gros}
  \zeta(2h-1) = \frac{(h-2)!}{(2h-2)!}\left\{ \frac{(4\pi)^{2h-1} h! |B_{2h}|}{2(2h)!}\right. \\
\left.-\sum_{m=1}^\infty \sigma_{1-2h}(m) e^{-2\pi m}\left(\sum_{\ell=0}^h \frac{(h+\ell-2)! \big[h(h-1)+\ell(\ell-1)\big]}{\ell!(h-\ell)!} (4\pi m)^{h-\ell} \right)\right\}.
\end{multline}
Grosswald provided the explicit form of \e{gros} in \cite{Gr75}.

We may use Theorem \ref{main2} to prove \e{gros}, and indeed generalize it, as follows. For $S := \left(\smallmatrix 0 & -1 \\ 1 & 0
\endsmallmatrix\right)$ we have $S i = i$ and $j(S,i)=i$. Hence by \e{wtk}, $E_k(i,s)=i^k E_k(i,s)$ and so
\begin{equation}\label{eisid}
  E^*_k(i,s) \equiv 0 \qquad \text{for} \qquad k \equiv 2 \bmod 4.
\end{equation}
Taking \e{eisid} with $s=1-h$ for $h \in \Z_{\gqs 1}$ and expanding with Theorem \ref{main2} gives the identity
\begin{multline} \label{big}
  \theta_k(h)+\theta_k(1-h) =  \\
  -\sum_{m=1}^\infty \frac{\sigma_{1-2h}(m)}{m^{1-h}} e^{-2\pi m}\left(
\sum_{u=0}^{h-1+k/2} \mathcal A^k_h(u) \cdot (4\pi m)^{-u+k/2}+\sum_{u=0}^{h-1-k/2} \mathcal A^{-k}_h(u) \cdot (4\pi m)^{-u-k/2}
 \right)
\end{multline}
for all $k \equiv 2 \bmod 4$.
We may assume that $k$ and $h$ are positive in \e{big} since it is invariant as $k \to -k$ and $h \to 1-h$ (except for replacing $\sigma_{1-2h}(m)/m^{1-h}$ by $\sigma_{2h-1}(m)/m^{h}$).
Also $\theta_k(1-h) = 0$ for $2\lqs h \lqs |k|/2$ by \e{thet}, so it is natural to consider \e{big} in the three cases: (i) $h=1$, (ii) $2\lqs h\lqs k/2$ and (iii) $k/2<h$. The first case easily gives
\begin{equation} \label{bigh1}
  \frac{\zeta(2) k}{2\pi} + \zeta(0) = -\sum_{m=1}^\infty \sigma_{-1}(m) e^{-2\pi m}
\sum_{r=1}^{k/2} (-1)^r \binom{k/2}{r} \frac{(4\pi m)^{r}}{(r-1)!}
\end{equation}
where the left side of \e{bigh1} is just $(\pi k-6)/12$.

\begin{prop} \label{biga} Let $h$ and $k$ be positive integers with $k \equiv 2 \bmod 4$ and $2\lqs h\lqs k/2$. Then
\begin{equation*}
  \zeta(2h) = -\pi^h \sum_{m=1}^\infty \frac{\sigma_{1-2h}(m)}{m^{1-h}} e^{-2\pi m}
\sum_{r=1-h}^{k/2} (-1)^r \binom{k/2-h}{k/2-r} \frac{(4\pi m)^{r}}{(h-1+r)!}.
\end{equation*}
\end{prop}
\begin{proof}
For these values of $h$ equation \e{big} implies
\begin{equation*}
  \theta_k(h) =  \\
  -\sum_{m=1}^\infty \frac{\sigma_{1-2h}(m)}{m^{1-h}} e^{-2\pi m}
\sum_{u=0}^{h-1+k/2} \mathcal A^k_h(u) \cdot (4\pi m)^{-u+k/2}.
\end{equation*}
Simplifying and letting $r=k/2-u$ gives the result.
\end{proof}

Examples of Proposition \ref{biga} with $(h,k)=(2,6)$, $(3,6)$ and $(3,10)$ are
\begin{align*}
  \zeta(4) & = \frac{8\pi}{3}\sum_{m=1}^\infty \sigma_{-3}(m)e^{-2\pi m}\left( -\pi^3 m^3 + \pi^4 m^4\right), \\
\zeta(6) & = \frac{8\pi}{15}\sum_{m=1}^\infty \sigma_{-5}(m)e^{-2\pi m}\left(  \pi^5 m^5 \right), \\
 \zeta(6) & = \frac{8\pi}{15}\sum_{m=1}^\infty \sigma_{-5}(m)e^{-2\pi m}\left(  \pi^5 m^5 -\frac 43 \pi^6 m^6 + \frac 8{21} \pi^7 m^7 \right).
\end{align*}
These may be compared with \e{euler}.

\begin{prop} \label{bigb} Let $h$ and $k$ be positive integers with $k \equiv 2 \bmod 4$ and $h > k/2$. Then
\begin{multline}\label{grosb}
  \zeta(2h-1) = \frac{(h-k/2-1)!}{(2h-2)!}\Bigg\{ \frac{(4\pi)^{2h-1} (h+k/2-1)! |B_{2h}|}{2(2h)!}-\sum_{m=1}^\infty \sigma_{1-2h}(m) e^{-2\pi m} \\
\times \sum_{\ell=1-k/2}^{h} \frac{k!(h+\ell-2)!}{(\ell+k/2-1)!(h-\ell)!}\left[ \binom{k/2+h-1}{k}+\binom{k/2+\ell-1}{k}\right] (4\pi m)^{h-\ell} \Bigg\}.
\end{multline}
\end{prop}
\begin{proof}
We use \e{big} and write the sums in parentheses as
\begin{equation*}
  \sum_{u=0}^{h-1+k/2} (4\pi m)^{-u+k/2} \left(\mathcal A^k_h(u) + \mathcal A^{-k}_h(u-k) \right)
\end{equation*}
with the understanding that $\mathcal A^{-k}_h(u-k)=0$ for $u<k$. Using the identity \e{binneg}
to make all the entries positive we obtain
\begin{align*}
  \mathcal A^k_h(u) & = (-1)^{k/2} u! \binom{k/2+h-1}{u}\binom{h-1-k/2+u}{u} \\
& = -\frac{1}{u!} \cdot \frac{(h-1-k/2+u)!}{(h-1+k/2-u)!} \cdot \frac{(h-1+k/2)!}{(h-1-k/2)!}.
\end{align*}
Similarly,
\begin{align*}
  \mathcal A^{-k}_h(u-k) & = (-1)^{k/2} \frac{[(u-k/2)!]^2}{(u-k)!} \binom{h-1}{u-k/2} \binom{h-1+u-k/2}{u-k/2} \\
& = -\frac{1}{(u-k)!} \cdot \frac{(h-1-k/2+u)!}{(h-1+k/2-u)!}
\end{align*}
for $k\lqs u \lqs h-1+k/2$. Rewrite the term $1/(u-k)!$  above as $\frac{k!}{u!} \binom{u}{k}$ to get the formula we want for $0\lqs u \lqs h-1+k/2$. Simplifying and  setting $\ell = u-k/2+1$ finishes the proof.
\end{proof}

Clearly, the Terras-Grosswald formula \e{gros} is the first case of Proposition \ref{bigb} when $k=2$. The following examples show the results of Lerch-Ramanujan and Terras-Grosswald, respectively, for $\zeta(7)$:
\begin{align*}
  \zeta(7) & = \frac{19 \pi^7}{56700} - 2\sum_{m=1}^\infty \sigma_{-7}(m) e^{-2\pi m}, \\
 \zeta(7) & = \frac{32 \pi^7}{70875} - 2\sum_{m=1}^\infty \sigma_{-7}(m) e^{-2\pi m}\left(1+2\pi m+\frac{28\pi^2 m^2}{15}+\frac{16\pi^3 m^3}{15}+\frac{16\pi^4 m^4}{45}\right).
\end{align*}
By comparison, Proposition \ref{bigb} with $h=4$ and $k=6$ gives
\begin{multline*}
 \zeta(7) = \frac{32 \pi^7}{4725} - 2\sum_{m=1}^\infty \sigma_{-7}(m) e^{-2\pi m}\\
\times\left(1+2\pi m+ 4\pi^2 m^2+\frac{16\pi^3 m^3}{3}+\frac{16\pi^4 m^4}{3}+\frac{64\pi^5 m^5}{15}+\frac{128\pi^6 m^6}{45}\right).
\end{multline*}

\section{Harmonic Eisenstein series} \label{harmo}
For each $k \in 2\Z$ set
\begin{equation}\label{emh}
\ei_k(z) := y^{-k/2}\frac{E^*_k(z,k/2)}{\theta_k(k/2)}.
\end{equation}
This definition  has poles when $k=0$  and so we interpret $\ei_0(z)$ as $\lim_{s\to 0}  E^*_0(z,s)/\theta_0(s)$.
These series are  studied by Pribitkin in \cite{Pr00} and have also appeared recently in \cite[Sect. 6.1.4]{BK}.
 Clearly $\ei_k(z)$ has holomorphic weight $k$. For even $k \geqslant 4$ we have from \e{ehol}, \e{compl} that $\ei_k(z) = E_k(z)$ and so we just recover the usual holomorphic Eisenstein series for these $k$. We may regard
 $\ei_k(z)$ for even $k \lqs 2$ as natural extensions of these holomorphic Eisenstein series.

 It follows from \e{alls} that
\begin{equation} \label{egex}
 \left(y^{-k/2} \circ \Delta_k \circ y^{k/2} - k/2(1-k/2)\right)\ei_k(z) = 0.
\end{equation}
Therefore
\begin{equation*}
  \left( -y^2 \left( \frac{\partial ^2}{\partial x^2} +  \frac{\partial ^2}{\partial y^2}\right) +ik y \left(\frac{\partial }{\partial x}+i\frac{\partial }{\partial y} \right)\right)\ei_k(z) = 0
\end{equation*}
which implies that $\ei_k(z)$ is a {\em harmonic Maass form}. See for example \cite{Ono,BFOR} for more on harmonic Maass forms and surveys of their important applications. We next obtain the Fourier expansion of $\ei_k(z)$ and see that the result can be formulated in a way that is valid for all even $k$.
It will also be apparent that for nonzero even $k \lqs 2$,  $\ei_k(z)$ differs from $E_k(z)$ in its extended Fourier expansion definition \e{eisk} and is not holomorphic.

\begin{theorem} \label{andr} Define $\varepsilon(k)$ to be $2$ if $k\gqs 0$ and $1$ if $k< 0$.
Then for  all $k \in 2\Z$ we have
\begin{multline} \label{negk}
 \ei_k(z) =
1  +
\frac{\varepsilon(k)}{\zeta(1-k)} \Bigg[\frac{\zeta(2-k)}{(2\pi i)^k} \left(\frac{y}{\pi}\right)^{1-k} \\
 +
\sum_{m =1}^\infty  \sigma_{k-1}(m) \ e^{2\pi i m z}
+
   \sum_{m =1}^\infty \sigma_{k-1}(m) \ e^{-2\pi i m \overline{z}} \sum_{u=0}^{-k}  \frac{(4\pi m y)^{u}}{u!}\Bigg].
\end{multline}
\end{theorem}
\begin{proof}
This is a straightforward application of Theorem \ref{main2} and we review the different cases.
For $k\gqs 2$ and $h=k/2$ then $h^*=k/2-1<k/2$ and
\begin{equation*}
  E^*_k(z,k/2)= \theta_k(k/2) y^{k/2}+\theta_k(1-k/2) y^{1-k/2} +
   \sum_{m =1}^\infty \frac{ \sigma_{k-1}(m)}{m^{k/2}} \ e^{2\pi i m z}  \mathcal A^k_{k/2}(0) \cdot (4\pi m y)^{k/2}.
\end{equation*}
We have $\mathcal A^k_{k/2}(0) = (-1)^{k/2}$ and hence
\begin{equation*}
  \ei_k(z) = 1+ \frac{\theta_k(1-k/2)}{\theta_k(k/2)} y^{1-k}+ \frac{(-4\pi)^{k/2}}{\theta_k(k/2)}
 \sum_{m =1}^\infty  \sigma_{k-1}(m) \ e^{2\pi i m z}  \qquad (k\gqs 2).
\end{equation*}
With \e{thet} we have $\theta_k(k/2) = \pi^{-k/2}(k-1)!\zeta(k)$. Also $\theta_k(1-k/2)$ is $ \zeta(0)=-1/2$ if $k=2$ and $0$ if $k>2$. Therefore
\begin{equation}\label{quasi}
  \ei_2(z) = 1 - \frac{3}{\pi y} -24 \sum_{m=1}^\infty \sigma_1(m) e^{2\pi i mz}
\end{equation}
and, using $(2\pi i)^{k} \zeta(1-k)=2 (k-1)!\zeta(k)$ for $k\gqs 2$,
\begin{equation}\label{ekep}
  \ei_k(z) = 1 +\frac{2}{\zeta(1-k)} \sum_{m=1}^\infty \sigma_{k-1}(m) e^{2\pi i mz} \qquad (k\gqs 4)
\end{equation}
as expected. Equations \e{quasi} and \e{ekep} match \e{negk} because $\zeta(2-k)$ is zero for even $k\gqs 4$ and the last sum, over $u$ in an empty range, vanishes for positive $k$. The important weight $2$ non-holomorphic form $\ei_2(z)$ was found by Hecke; see for example \cite[Sect. 2.3]{Za} where it is denoted $E^*_2(z)$. Hence $-3/(\pi y)$ must be added to $E_2$ in \e{eisk} so that it transforms with weight $2$.

For $k=0$ we have
\begin{equation*}
  \ei_0(z) = \lim_{s\to 0} \frac{E^*_0(z,s)}{\theta_0(s)} = \lim_{s\to 0} \frac{\theta_0(s) +O(1)}{\theta_0(s)} = \lim_{s\to 0} \frac{-1/(2s) +O(1)}{-1/(2s) +O(1)} =1.
\end{equation*}
Since $\zeta(s)$ has a pole at $s=1$ we take $\varepsilon(k)/\zeta(1-k)$ to be $0$ when $k=0$ and so \e{negk} confirms that $\ei_0(z)=1$.

Lastly we assume $k\lqs -2$ and $h=k/2$ so that $h^*=-k/2$. Theorem \ref{main2} implies that
\begin{multline} \label{tiac}
  \ei_k(z) = 1+ \frac{\theta_k(1-k/2)}{\theta_k(k/2)} y^{1-k}+ \frac{(4\pi)^{k/2}}{\theta_k(k/2)}
 \sum_{m =1}^\infty  \sigma_{k-1}(m) \ e^{2\pi i m z} \mathcal A^k_{k/2}(0)  \\
  + \frac{1}{\theta_k(k/2)}
 \sum_{m =1}^\infty  \frac{\sigma_{k-1}(m)}{(my)^{k/2}} \ e^{-2\pi i m \overline z} \sum_{u=0}^{-k} \mathcal A^{-k}_{k/2}(u) \cdot (4\pi m y)^{-u-k/2}
\qquad (k\lqs -2).
\end{multline}
With \e{defak2}, \e{binneg} and writing $-k$ as $|k|$ for clarity
\begin{align*}
  \mathcal A^k_{k/2}(0) & = [(|k|/2)!]^2 \binom{-|k|/2-1}{|k|/2}\binom{|k|/2}{|k|/2} = (-1)^{k/2} [(|k|/2)!]^2 \binom{|k|}{|k|/2} = (-1)^{k/2} |k|!,\\
  \mathcal A^{-k}_{k/2}(u) & = (-1)^{u+k/2} u! \binom{-1}{u}\binom{|k|}{u} = (-1)^{k/2} u! \binom{|k|}{u} = (-1)^{k/2} \frac{|k|!}{(|k|-u)!}.
\end{align*}
 With \e{thet} we see
\begin{equation*}
  \theta_k(k/2)  =(-4\pi)^{k/2} |k|! \zeta(1-k), \qquad \theta_k(1-k/2)  =\pi^{k/2-1} |k|! \zeta(2-k).
\end{equation*}
The proof is finished by substituting these calculations into \e{tiac} and simplifying.
\end{proof}


\SpecialCoor
\psset{griddots=5,subgriddiv=0,gridlabels=0pt}
\psset{xunit=1.8cm, yunit=1.2cm, runit=1.8cm}
\psset{linewidth=1pt}
\psset{dotsize=4.5pt 0,dotstyle=*}
\psset{arrowscale=1.5,arrowinset=0.3}

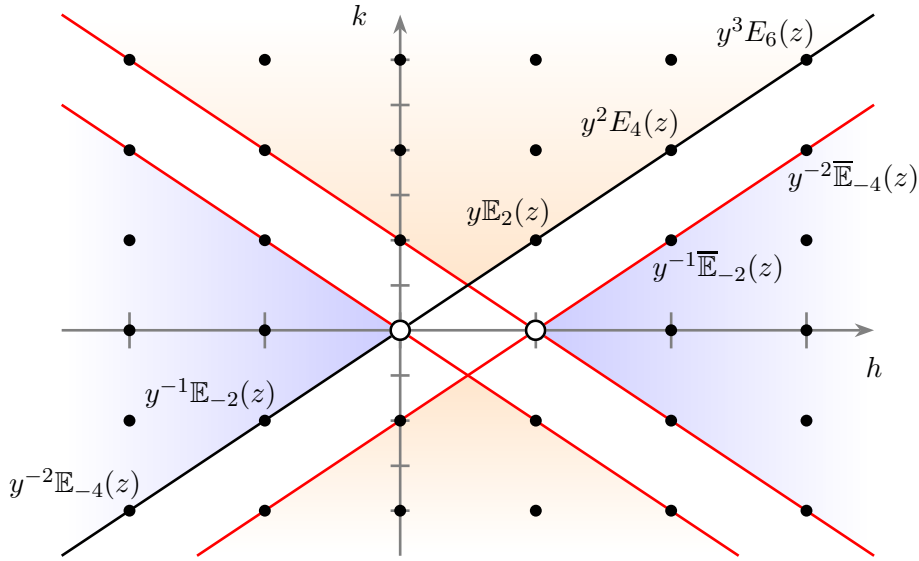
\begin{figure}[ht]
\centering
\begin{pspicture}(-3,-2.5)(4,3.5) 

\pscustom[fillstyle=gradient,gradangle=0,linecolor=white,gradmidpoint=1,gradbegin=white,gradend=lightorange,gradlines=100]{%
\pspolygon*[linecolor=lightblue](-2.5,3.5)(0.5,0.5)(3.5,3.5)} 
\pscustom[fillstyle=gradient,gradangle=180,linecolor=white,gradmidpoint=1,gradbegin=white,gradend=lightorange,gradlines=100]{%
\pspolygon*[linecolor=lightblue](-1.5,-2.5)(0.5,-0.5)(2.5,-2.5)} 
\pscustom[fillstyle=gradient,gradangle=270,linecolor=white,gradmidpoint=1,gradbegin=white,gradend=lightblue,gradlines=100]{%
\pspolygon*[linecolor=lightblue](3.5,2.5)(1,0)(3.5,-2.5)} 
\pscustom[fillstyle=gradient,gradangle=90,linecolor=white,gradmidpoint=1,gradbegin=white,gradend=lightblue,gradlines=100]{%
  \pspolygon[linecolor=lightblue](-2.5,2.5)(0,0)(-2.5,-2.5)} 

\psline[linecolor=gray]{->}(-2.5,0)(3.5,0)
\psline[linecolor=gray]{->}(0,-2.5)(0,3.5)


\multirput(-2,-0.2)(1,0){6}{\psline[linecolor=gray](0,0)(0,0.4)}
\multirput(-0.07,-2)(0,0.5){11}{\psline[linecolor=gray](0,0)(0.14,0)}

\psline[linecolor=red](-1.5,-2.5)(3.5,2.5)
\psline[linecolor=red](-2.5,3.5)(3.5,-2.5)
\psline[linecolor=red](-2.5,2.5)(2.5,-2.5)
\psline(-2.5,-2.5)(3.5,3.5)

\multirput(-2,-2)(1,0){6}{\psdot(0,0)}
\multirput(-2,-1)(1,0){6}{\psdot(0,0)}
\multirput(-2,-0)(1,0){6}{\psdot(0,0)}
\multirput(-2,1)(1,0){6}{\psdot(0,0)}
\multirput(-2,2)(1,0){6}{\psdot(0,0)}
\multirput(-2,3)(1,0){6}{\psdot(0,0)}

\pscircle*[linecolor=white,linewidth=1pt](0,0){0.08}
\pscircle[linecolor=black,linewidth=1pt](0,0){0.08}
\pscircle*[linecolor=white,linewidth=1pt](1,0){0.08}
\pscircle[linecolor=black,linewidth=1pt](1,0){0.08}

\rput(3.5,-0.4){$h$}
\rput(-0.3,3.5){$k$}
\rput(0.8,1.3){$y\ei_2(z)$}
\rput(1.7,2.3){$y^2 E_4(z)$}
\rput(2.7,3.3){$y^3 E_6(z)$}
\rput(-1.4,-0.7){$y^{-1}\ei_{-2}(z)$}
\rput(-2.4,-1.7){$y^{-2}\ei_{-4}(z)$}

\rput(2.35,0.7){$y^{-1}\overline{\ei}_{-2}(z)$}
\rput(3.35,1.7){$y^{-2}\overline{\ei}_{-4}(z)$}



\end{pspicture}
\caption{Eisenstein series on the main diagonal}
\label{cfig}
\end{figure}


As  seen in Figure \ref{cfig}, the harmonic Eisenstein series $\ei_k(z)$ times  $y^{k/2}$ may be used as representatives for the spaces generated by $E^*_k(z,h)$ on the main diagonal $h=k/2$. We indicated on the left of Figure \ref{bfig} that, with \e{raise}, \e{lower}, the raising operators naturally move up the lattice and the lowering operators move down. Thus we see that $y^{k/2}\ei_k(z)$ for $k/2 \in \Z_{\gqs 1}$ generates the upper triangle of Figure \ref{cfig} by applying the raising operators -- recall that the left and right sides of this triangle are equal by \e{fek}. We may not leave the upper triangle by means of the lowering operator since $L_k E^*_k(z,k/2)=0$ for $k/2\in \Z_{\gqs 2}$ (though $L_2 E^*_2(z,1)=1/2$). The left and right triangles of Figure \ref{cfig} are generated by $y^{k/2}\ei_k(z)$ for $k/2 \in \Z_{\lqs -1}$ and the raising operators (or alternatively $y^{k/2}\overline{\ei}_k(z)$ for $k/2 \in \Z_{\lqs -1}$ and the lowering operators).

\subsection{The holomorphic part of $\ei_k(z)$}
In this subsection we assume that $k$ is even and negative. The holomorphic series in \e{negk} may be rewritten as
\begin{equation}\label{xxx}
  U_k(z):=\sum_{m =1}^\infty  \sigma_{k-1}(m) \ e^{2\pi i m z} = (2\pi i)^{1-k}\sum_{m =1}^\infty \frac{\sigma_{1-k}(m)}{(2\pi i m)^{1-k}} \ e^{2\pi i m z}.
\end{equation}
Hence it is a constant times a $(1-k)$ fold integral of $E_{2-k}(z)-1$, as noted by Pribitkin in \cite{Pr00}:
\begin{equation*}
  U_k(z)=\frac{(1-k)!\zeta(2-k)}{2\pi i} \int_{i\infty}^z \cdots \int_{i\infty}^{z_2} \int_{i\infty}^{z_1} (E_{2-k}(z_0)-1)\, dz_0 \, dz_1 \, \cdots \, dz_{|k|}.
\end{equation*}
Integrating by parts shows the alternate expression
\begin{equation} \label{inz}
  U_k(z) = \frac{(1-k)\zeta(2-k)}{2\pi i} \int_{i\infty}^z (E_{2-k}(w)-1)(w-z)^{-k} \, dw.
\end{equation}

The integral $U_k(z)$ is a type of {\em Eichler integral} and  transforms with weight $k$ except for an additional rational function of $z$. Following \cite[(7)]{Pr00} we have:

\begin{lemma}
Let $\g =\left(\smallmatrix * & * \\ c & *
\endsmallmatrix\right) \in \G$ with $c\neq 0$. Then
\begin{multline} \label{transi}
  j(\g,z)^{-k} U_k(\g z) - U_k(z) = \frac{(1-k)\zeta(2-k)}{2\pi i}\Bigg\{\frac 1{c^{1-k}(1-k)}\left(j(\g,z)^{1-k}+\frac 1{j(\g,z)} \right) \\
  +\int^{i \infty}_{\g^{-1}(i \infty)}\left(E_{2-k}(w)-1-\frac 1{j(\g,w)^{2-k}} \right) (w-z)^{-k}\, dw\Bigg\}.
\end{multline}
\end{lemma}
\begin{proof}
Recall the identity
\begin{equation*}
  j(\g,z)(w-\g z) = j(\g^{-1},w)(\g^{-1}w-z)
\end{equation*}
for all $z,w\in \H$ and any invertible $\g$. Then
\begin{align*}
 \frac{2\pi i}{(1-k)\zeta(2-k)} j(\g,z)^{-k} U_k(\g z) & =  \int_{i\infty}^{\g z} (E_{2-k}(w)-1)\left[ j(\g,z)(w-\g z)\right]^{-k} \, dw \\
   &  =  \int_{i\infty}^{\g z} (E_{2-k}(w)-1)\left[ j(\g^{-1},w)(\g^{-1}w-z) \right]^{-k} \, dw \\
 &  =  \int_{\g^{-1}(i\infty)}^{z} (E_{2-k}(\g u)-1)\left[ j(\g^{-1},\g u)(u-z) \right]^{-k} j(\g^{-1},\g u)^2 \, du \\
&  =  \int_{\g^{-1}(i\infty)}^{z} \left(E_{2-k}(u)-\frac{1}{j(\g,u)^{2-k}} \right)(u-z)^{-k} \, du.
\end{align*}
Rewrite this last integral as
\begin{multline*}
  \frac{2\pi i}{(1-k)\zeta(2-k)}U_k(z) +\int^{i \infty}_{\g^{-1}(i \infty)}\left(E_{2-k}(w)-1-\frac 1{j(\g,w)^{2-k}} \right) (w-z)^{-k}\, dw
\\
-\int_{i\infty}^{z} \frac{(w-z)^{-k}}{j(\g,w)^{2-k}}\, dw + \int_{\g^{-1}(i\infty)}^{z} (w-z)^{-k}\, dw
\end{multline*}
and we obtain the result. Note that  the Fourier expansion \e{eisk} implies  $E_{2-k}(z)=1+O(e^{-2\pi y})$ as $y=\Im(z) \to \infty$. Also if $u\to \g^{-1}(i\infty)$ then we may write $u=\g z$ with $\Im(z) \to \infty$. As a consequence of the automorphy of $E_{2-k}$,
\begin{equation*}
  E_{2-k}(u) = j(\g,u)^{k-2}(1+O(e^{-2\pi y})) \quad \text{as} \quad u \to \g^{-1}(i\infty).
\end{equation*}
It follows that all the integrals in the proof are absolutely convergent.
\end{proof}

Letting $\g = S$ in \e{transi} produces
\begin{multline} \label{transi2}
  z^{-k} U_k\left(\frac{-1}{z}\right) - U_k(z) = \frac{(1-k)\zeta(2-k)}{2\pi i}\Bigg\{\frac 1{1-k}\left(z^{1-k}+\frac 1{z} \right)
\\
  +\int^{i \infty}_{0}\left(E_{2-k}(w)-1-\frac 1{w^{2-k}} \right) (w-z)^{-k}\, dw\Bigg\}.
\end{multline}
Except for a term $1/z$, the right side of \e{transi2} is a polynomial in $z$ and as in \cite[Sect. 2]{Za91} we may compute it explicitly.

\begin{prop} \label{cos}
For $z\in \H$ and even $k\lqs -2$,
\begin{equation*}
   2\left(z^k U_k(z)
- U_k(-1/z) \right)
  = (2\pi i)^{1-k} \sum_{\substack{u,v \in \Z_{\gqs 0} \\u+v=1-k/2}}  \frac{B_{2u}}{(2u)!}\frac{B_{2v}}{(2v)!} z^{1-2v}
+
     \left( 1-z^k\right)\zeta(1-k).
\end{equation*}
\end{prop}
\begin{proof}
With $n=2-k \gqs 4$, the integral on the right of \e{transi2} is
\begin{equation} \label{tar}
  \sum_{r=0}^{n-2} \binom{n-2}{r}(-z)^{n-2-r} i^{r+1} \int_0^\infty \left( E_n(iy)-1-\frac{1}{i^n y^n}\right) y^r\, dy.
\end{equation}
Let $c_n:=(2\pi i)^n/(\G(n)\zeta(n))$. For $\Re(s)$ large enough, the completed $L$-function associated to $E_n$ is
\begin{equation*}
  L^*(E_n,s):=\int_0^\infty \left( E_n(it)-1\right) t^{s-1}\, dt = \frac{c_n \G(s)}{(2\pi)^s} \sum_{m = 1}^\infty \frac{\sigma_{n-1}(m)}{m^s}.
\end{equation*}
Rewriting as
\begin{equation} \label{leksfe}
  L^*(E_n,s) = \int_1^\infty \left( E_n(it)-1\right) t^{s-1}\, dt + \int_0^1 \left( E_n(it)-\frac{1}{i^n t^n}\right) t^{s-1}\, dt -\frac{1}{s}-\frac{1}{i^n(n-s)}
\end{equation}
shows that $L^*(E_n,s)$ has a meromorphic continuation to all $s\in \C$ with poles only at $s=0,n$. Combining $E_n(i/y)=i^ny^n E_n(iy)$ with \e{leksfe} also gives the functional equation
\begin{equation}\label{lek}
  L^*(E_n,n-s) = (-1)^{n/2} L^*(E_n,s).
\end{equation}
For $0<\Re(s)<n$ we see that
\begin{equation}\label{wer}
  \int_0^\infty \left( E_n(iy)-1-\frac{1}{i^n y^n}\right) y^{s-1}\, dy = L^*(E_n,s)
\end{equation}
since both sides agree with the right side of \e{leksfe}. Therefore \e{tar} equals
\begin{equation} \label{tar2}
  \sum_{r=0}^{n-2} \binom{n-2}{r}(-z)^{n-2-r} i^{r+1} L^*(E_n,r+1).
\end{equation}
For $\Re(s)$ large,
\begin{equation*}
  \sum_{m = 1}^\infty \frac{\sigma_{n-1}(m)}{m^s} = \sum_{d=1}^\infty \sum_{r=1}^\infty \frac{d^{n-1}}{(rd)^s} = \zeta(s)\zeta(s+1-n).
\end{equation*}
Recall that $\zeta(1-n)=(-1)^{n+1}B_n/n$ for $n\in \Z_{\gqs 1}$. This means $\zeta(0)=-1/2$ and $\zeta(n)=0$ for all negative even $n$.
We have $$L^*(E_n,n-1)=c_n (n-2)! (2\pi)^{1-n} \zeta(n-1)\zeta(0)$$ and by \e{lek}, $L^*(E_n,1)$ is $(-1)^{n/2}$ times this value. For integers $r$ with $2\lqs r\lqs n-2$,
\begin{equation*}
  L^*(E_n,r)=c_n (r-1)! (2\pi)^{-r} \zeta(r)\zeta(r+1-n)
\end{equation*}
and this vanishes for odd $r$. For even $r$  in this range
\begin{equation*}
  L^*(E_n,r)= (-1)^{n+r}c_n \frac{(r-1)!}{(2\pi)^{r}} \frac{(2\pi)^r |B_r|}{2\cdot r!}\frac{B_{n-r}}{(n-r)!}.
\end{equation*}
Assembling these calculations and simplifying completes the proof.
\end{proof}
Proposition \ref{cos} confirms \e{bern2} for even $k\lqs -2$. For even $k\gqs 2$, the identity \e{bern2} follows from the fact that $\ei_k(z)$ has holomorphic weight $k$. The weight $k=0$ case may be shown with the Kronecker limit formula as in \cite[Sect. 3]{DIT}, for example.

As discussed in \cite[pp. 466-467]{Pr00} and \cite[Remark 5.2]{BS17}, (and returning to $k$ being even and negative), adding a $(1-k)$ fold integral of the constant term of $E_{2-k}(z)$ to $U_k(z)$ gives it a neater transformation property. Under the action of $S=\left(\smallmatrix 0 & -1 \\ 1 & 0
\endsmallmatrix\right)$ it will now transform with weight $k$ up to a polynomial instead of a rational function -- this is the period polynomial of  $E_{2-k}(z)$ with respect to $S$.

These ideas generalize and the period polynomial with respect to $S$ of any modular form $f$ has coefficients that may be expressed in terms of $L^*(f,r)$ with integer $r$ satisfying $1\lqs r \lqs k-1$. This is due to work of Grosswald, Razar,  Weil and others; see Sections 4 and 5 of \cite{Mur11} or Section 5 of \cite{BS17} and the contained references. There is also a nice detailed treatment in Sections 11.5-11.7 of \cite{CS17}, though without references.

\subsection{The non-holomorphic part of $\ei_k(z)$} \label{nhei}
As in \e{jkz}, write the non-holomorphic series in \e{negk} as
\begin{equation} \label{jkz2}
   V_k(z):= \sum_{m=1}^\infty \sigma_{k-1}(m) e^{-2\pi i m \overline{z}}\sum_{u=0}^{-k}  \frac{(4\pi m y)^{u}}{u!}.
\end{equation}
It is identically zero for $k$ positive.
The finite sum in \e{jkz2} may be expressed, for $x>0$ and $m\in \Z_{\gqs 0}$, as
\begin{align}\label{pbk}
  \sum_{u=0}^m \frac{x^u}{u!} & = \frac{x^{m+1}}{m!}  \int_0^\infty (t+1)^m e^{-xt}\, dt, \\
   & = \frac{e^x}{m!}  \int_x^\infty t^{m} e^{-t}\, dt = \frac{e^x}{m!}  \G(m+1,x). \label{pbk2}
\end{align}
Using \e{pbk} or \e{pbk2} we obtain a similar identity to \e{inz},
\begin{equation*}
  V_k(z) = \frac{(1-k)\zeta(2-k)}{2\pi i}\int_{i\infty}^{-\overline{z}} (E_{2-k}(w)-1)(w+z)^{-k} \, dw
\end{equation*}
for even $k \lqs 0$.
This is equivalent to the expression in \cite[p. 466]{Pr00}. With \e{pbk2} we may also write
\begin{equation*}
  V_k(z) = \sum_{m=1}^\infty \frac{\sigma_{k-1}(m)}{|k|!} e^{-2\pi i m z} \cdot \G(1-k,4\pi my).
\end{equation*}
In this formulation using the incomplete $\G$ function, Theorem \ref{andr} matches the general Fourier development expected for a harmonic Maass form with polynomial growth at the cusps, as given in \cite[Lemma 4.4]{LR16} for example.

\begin{proof}[Proof of Theorem \ref{jkt}] We have
$
  z^k \ei_k(z)
- \ei_k(-1/z) = 0
$
since $\ei_k(z)$ has holomorphic weight $k$.
Expressing this relation with the Fourier expansion from Theorem \ref{andr} shows
\begin{multline} \label{sno}
  2\left(z^k V_k(z)
- V_k(-1/z) \right)
  =
\frac{2\zeta(2-k)}{(2\pi i)^k} \left( \frac y\pi\right)^{1-k}\left(|z|^{2k-2}-z^k \right)
\\
+2\frac{\zeta(1-k)}{\varepsilon(k)}\left(1- z^k\right)- 2\left(z^k U_k(z)
- U_k(-1/z) \right)
\end{multline}
for all nonzero $k\in 2\Z$. Then substituting \e{bern2} into \e{sno} and simplifying gives the result when $k\neq 0$. For $k=0$ we have
\begin{equation} \label{slal}
  2\left( U_0(z)
- U_0(-1/z) \right) = \pi i(z+1/z)/6-\pi i/2+\log z
\end{equation}
by \e{bern2}. Clearly $V_0(z)=\overline{U_0(z)}$ and so, by conjugating \e{slal},
\begin{equation} \label{slal2}
  2\left( V_0(z)
- V_0(-1/z) \right) = -\pi i(\overline z+1/\overline z)/6+\pi i/2+\overline{\log z}.
\end{equation}
Verify that the theorem gives \e{slal2} for $k=0$.
\end{proof}

Letting $z=i$ and $k=2-2h$ in Theorem \ref{jkt} gives the following companion identity to \e{lerch}
\begin{multline}\label{lerchxx}
  \zeta(2h-1) = \frac{(4\pi)^{2h-1}|B_{2h}|}{(2h)!}+\frac{(2\pi)^{2h-1}}{2} \sum_{\substack{u,v \in \Z_{\gqs 0} \\u+v=h}} (-1)^u \frac{B_{2u}}{(2u)!}\frac{B_{2v}}{(2v)!}\\
 -2\sum_{m=1}^\infty \sigma_{1-2h}(m) e^{-2\pi m}\sum_{u=0}^{2h-2} \frac{(4\pi m)^u}{u!},
\end{multline}
for even $h \gqs 2$. The Bernoulli numbers sum appears with opposite signs in  \e{lerch} and \e{lerchxx} so that adding them gives the simpler identities \e{lerchxxx} and \e{lerchxxx2} we saw in the introduction. It may easily be checked that Proposition \ref{bigb} for $k=2h-2$ also gives \e{lerchxxx} and \e{lerchxxx2}.

\subsection{Ramanujan polynomials}
Alternatively, we may obtain the Bernoulli numbers sum without the $\zeta(2h-1)$ term by taking the difference of \e{lerch} and \e{lerchxx}. More generally we may consider the polynomials
\begin{equation*}
 \mathcal R_k(z) := \sum_{\substack{u,v \in \Z_{\gqs 0} \\u+v=1-k/2}}  \frac{B_{2u}}{(2u)!}\frac{B_{2v}}{(2v)!} z^{2u}
\end{equation*}
appearing in \e{bern2} and named the {\em Ramanujan polynomials} in \cite{Mur11,MSW11}. We continue our practice of using the weight $k$ as the  index; $\mathcal R_k(z)$ corresponds to $R_{1-k}(z)$ in the cited papers. The simplest nonzero polynomials are $\mathcal R_2(z)=1$, $\mathcal R_0(z)=(1+z^2)/12$ and $\mathcal R_{-2}(z)=(-1+5z^2-z^4)/720$. Detailed properties of the zeros of $\mathcal R_k(z)$ are proved in \cite{MSW11} including: for even $k \lqs 0$ all nonreal zeros lie on the unit circle and, for even $k \lqs -2$, there are exactly four real zeros and they lie in the interval $[-2.2,2.2]$.

A consequence of \e{bern2} is that for negative even $k$
\begin{equation}\label{rmz}
  (2\pi i)^{1-k} \frac{\mathcal R_k(z)}{z(z^{-k}-1)} + \zeta(1-k) = \frac{2\left(U_k(z)-z^{-k}U_k(-1/z) \right)}{z^{-k}-1}.
\end{equation}
They further show in  \cite{MSW11} that for each even $k\lqs -8$, there is a zero $\alpha$ of $\mathcal R_k(z)$ with $\alpha \in \H$, $|\alpha|=1$ and $\alpha^{-k}\neq 1$. Hence  $\zeta(1-k)$ equals the right hand side of \e{rmz} evaluated at the algebraic number $\alpha$. This gives circumstantial evidence that $\zeta(1-k)$ is transcendental because  \cite[Thm. 2.1]{Mur11} says that as $\beta \in \H$ runs over all algebraic numbers, with $\beta^{-k}\neq 1$, the right  of \e{rmz} evaluated at  $\beta$ has an  algebraic value at most once. This unique algebraic value would have to be $\zeta(1-k)$ for $\zeta(1-k)$ to be algebraic.

Theorem \ref{jkt} gives a new expression for the Ramanujan polynomials in terms of the non-holomorphic components $V_k(z)$. Similarly to \e{rmz} we may write, for even $k<0$,
\begin{multline}\label{rmz2}
  -(2\pi i)^{1-k} \frac{\mathcal R_k(z)}{z(z^{-k}-1)} + \zeta(1-k) = \frac{2\left(V_k(z)-z^{-k}V_k(-1/z) \right)}{z^{-k}-1}
  \\ +\frac{B_{2-k}}{2(2-k)!}(4\pi y)^{1-k}\frac{z^{-k}|z|^{2k-2}-1}{z^{-k}-1}.
\end{multline}
Further results and conjectures related to the Ramanujan polynomials are discussed in \cite[Sect. 7]{BS17}.

\section{An inner product formula} \label{inner}
For two functions $f$ and $g$ on $\H$ with holomorphic weight $k$, their Petersson inner product is
\begin{equation}\label{pip}
 \s{f}{g} :=\int_{\GH} y^{k} f(z) \overline{g(z)} \, \frac{dx dy}{y^2}.
\end{equation}
This converges, for example, if   one of the functions has exponential decay at infinity and one has at most polynomial growth.
The following proposition links  a convolution $L$-series with the inner product of a cusp form $f$ and a product of Eisenstein series. This is an important step in the method of Diamantis and the author in \cite{DO10} and we give a new direct proof here based on Theorem \ref{allk2}.

\begin{prop} \label{RankinS} Let $f(z)=\sum_{n=1}^\infty a_f(n) e^{2\pi i nz}$ be a weight $k$ cusp form  and suppose $k_1, k_2 \in
2 \mathbb Z$
satisfy $k_1+k_2=k$ and $k_2 \geqslant 0$. Then for all $u, v \in \mathbb C$ satisfying
\begin{equation}\label{sineq}
  -\Re(u)-k/2+1/2<|\Re(v)-1/2|<\Re(u)-k_2/2-7/2
\end{equation}
we have
\begin{multline} \label{olym} \Bigl\langle f, y^{-k/2}E^*_{k_1}(\cdot, \overline u)
E^*_{k_2}(\cdot, \overline v) \Bigr\rangle \\
=
\left[(-1)^{k_2/2} 2 \pi^{k/2} \zeta(2u) \frac{\G(|k_1|/2+u)}{\G(k_1/2+u)} \frac{ \G(s)
\G(w)}{(2 \pi)^{s+w}} \right]\sum_{n=1}^{\infty}\frac{a_f(n) \sigma_{2v-1}(n)}{n^{w}}
\end{multline}
 where
\begin{equation}\label{sw}
s=u-v+k/2, \quad w=u+v+k/2-1.
\end{equation}
\end{prop}
\begin{proof} Note that the right inequality in \e{sineq} implies $\Re(u)>7/2$ and the left inequality  implies  $\Re(s)>0$ and $\Re(w)>0$. With \e{connj}, the left side of \e{olym} is
\begin{equation} \label{curl}
  \int_{\GH} y^{k/2} f(z) E^*_{-k_1}(z,  u)
E^*_{-k_2}(z,  v) \, \frac{dx dy}{y^2}.
\end{equation}
For $\Re(u)>1$ we may unfold $E_{-k_1}(\cdot,  u)$ in \e{curl} using its series expansion (\ref{defe}) to produce
\begin{equation}\label{int}
\theta_{k_1}(u)\int_0^\infty \int_0^1 y^{k/2+u-2} f(z) E^*_{-k_2}(z, v) \, dxdy.
\end{equation}
We need conditions on $u$ and $v$ for \e{int} to converge. In fact we require the convergence of \e{int} when $f(z)$ and $E^*_{-k_2}(z, v)$ are replaced by their Fourier expansions, bounded termwise in absolute value.
Employing Hecke's bound $a_f(n) \ll n^{k/2}$ we have, with \e{bbgd},
\begin{equation*}
  f(z) \ll \sum_{n=1}^\infty n^{k/2} e^{-2\pi n y} \ll (1+y^{-k/2-1})e^{-2\pi y}.
\end{equation*}
Corollary \ref{bnthc} gives the bounds we need for $E^*_{-k_2}(z, v)$ and we see that the integrand in \e{int} is bounded by
\begin{equation*}
  y^{k/2+\Re(u)-2} \cdot y^{-k/2-1} \cdot y^{-|\Re(v)-1/2|-k_2/2-3/2} \quad \text{as} \quad y\to 0.
\end{equation*}
Hence the desired convergence of \e{int} is ensured by the right inequality of \e{sineq}.

 Writing the Fourier expansion of $E^*_{-k_2}(z, v)$ as $\sum_{m \in \Z} e_{-k_2}(m;y,v)  e^{2\pi i m x}$,
the expression \e{int} becomes
$$
\theta_{k_1}(u)\sum_{n=1}^\infty a_f(n) \int_0^\infty  e^{-2\pi ny} e_{-k_2}(-n;y,v) y^{k/2+u-2} \, dy
$$
after integrating with respect to $x$.
Theorem \ref{allk2} tells us that
$$
e_{-k_2}(-n;y,v) = \frac{1}{2^{k_2-1}} \frac{ \sigma_{2v-1}(n)}{n^v}  \sum_{r=-k_2/2}^{k_2/2} P^{k_2/2}_r\bigl(4\pi n y \bigr) (n y)^{1/2} K_{v+r-1/2}(2 \pi n y)
$$
and so we need the integral
\begin{multline*}
\int_0^\infty  e^{-2\pi ny} P^{k_2/2}_r (4\pi n y ) (ny)^{1/2} K_{v+r-1/2}(2 \pi ny) y^{k/2+u-2} \, dy\\
=\sum_{\ell=|r|}^{k_2/2} \psi^{k_2/2}_r(\ell)\int_0^\infty  e^{-2\pi ny} (-4\pi n y )^\ell (ny)^{1/2} K_{v+r-1/2}(2 \pi ny) y^{k/2+u-1} \, \frac{dy}{y}\\
=\sum_{\ell=|r|}^{k_2/2} \frac{(-2)^\ell \psi^{k_2/2}_r(\ell) }{(2\pi)^{1/2} (2\pi n)^{k/2+u-1}} \int_0^\infty  e^{-2\pi ny} (2\pi n y)^{k/2+u-1/2+\ell} K_{v+r-1/2}(2 \pi ny) \, \frac{dy}{y}
\end{multline*}
with $\psi^{k_2/2}_r(\ell)$ the coefficient of $(-x)^\ell$ in $P^{k_2/2}_r(x)$. This last integral equals
\begin{equation*}
  \int_0^\infty  e^{-t} t^{k/2+u-1/2+\ell} K_{v+r-1/2}(t) \, \frac{dt}{t}
  =\frac{2^{-k/2+1/2-u-\ell} \pi^{1/2} \G(s+\ell-r)\G(w+\ell+r)}{\G(k/2+u+\ell)}
\end{equation*}
provided $\Re(s)+\ell-r>0$ and $\Re(w)+\ell+r>0$; see for example \cite[p. 205]{IwSp}. These conditions hold since $\Re(s),$ $\Re(w)>0$.

Putting our results together, simplifying and reordering the summation shows that the left side of \e{olym} equals
\begin{multline} \label{mul}
\frac{k_2! \theta_{k_1}(u)}{2^{k_2} (4\pi)^{k/2+u-1}}\sum_{n=1}^\infty \frac{a_f(n) \sigma_{2v-1}(n) }{n^{w}}
 \sum_{\ell=0}^{k_2/2}  \frac{( -1)^\ell }{\G(k/2+u+\ell) (k_2/2 -\ell)!} \\
 \times\sum_{r=-\ell}^\ell
 \frac{ \G(s+\ell-r)\G(w+\ell+r)}{(\ell-r)!(\ell+r)!}.
\end{multline}
The general binomial theorem implies that
$$
\sum_{j=0}^\infty \frac{\G(a+j)}{\G(a) j!} x^j =\frac{1}{(1-x)^a}
$$
for all $|x|<1$ and  $a>0$. Hence, expanding both sides of $(1-x)^{-a} (1-x)^{-b} = (1-x)^{-a-b}$ proves
\begin{equation}\label{gam}
\sum_{m+n=\ell} \frac{\G(a+m)\G(b+n)}{\G(a)\G(b)\ m! n!} = \frac{\G(a+b+\ell)}{\G(a+b)\ \ell!}
\end{equation}
and \e{gam} is valid for all $a,b \in \C$ by meromorphic continuation. Thus the innermost sum in (\ref{mul}) evaluates to
$
\G(s) \G(w) \G(s+w+2\ell)/(\G(s+w) (2\ell)!)
$
and (\ref{mul}) becomes
\begin{equation*}
  \frac{k_2! \theta_{k_1}(u)}{2^{k_2} (4\pi)^{k/2+u-1}}\frac{\G(s) \G(w)}{\G(s+w)}\sum_{n=1}^\infty \frac{a_f(n) \sigma_{2v-1}(n) }{n^{w}}
 \sum_{\ell=0}^{k_2/2} ( -1)^\ell \frac{ \G(s+w+2\ell) }{\G(k/2+u+\ell) (k_2/2 -\ell)! (2\ell)!}.
\end{equation*}
The next lemma  computes the above sum over $\ell$ (writing $s+w$ as $k-1+2u$).

\begin{lemma}
For even $k,k_2 \geqslant 0$ and all $u \in \C$
\begin{equation}\label{summ}
\sum_{\ell=0}^{k_2/2} ( -1)^\ell \frac{\G(k-1+2u+2\ell) }{\G(k/2+u+\ell) \ (k_2/2 -\ell)! \ (2\ell)!}=\frac{(2i)^{k_2}}{k_2!} \frac{\G(k-1+2u)}{\G(k/2-k_2/2 +u)}.
\end{equation}
\end{lemma}
\begin{proof}
Assume first that $u$ is a positive integer and let $d=k/2+u-1$. The left side of (\ref{summ}) is
$$
\sum_{\ell=0}^{k_2/2} ( -1)^\ell \frac{(2d+2\ell)!}{(2\ell)! (k_2/2-\ell)!(d+\ell)!} = \frac{(2d)!}{(k_2/2+d)!} \sum_{\ell} (-1)^\ell \binom{2d+2\ell}{2d} \binom{d+k_2/2}{d+\ell}.
$$
Replace $\ell$ by $\ell-d$ in this sum to get
$$
\frac{(-1)^d (2d)!}{(k_2/2+d)!} \sum_{\ell} (-1)^\ell \binom{2\ell}{2d} \binom{d+k_2/2}{\ell}
=\frac{(2i)^{k_2} (2d)!}{(d-k_2/2)! k_2!}=\frac{(2i)^{k_2} \G(k-1+2u)}{\G(k/2-k_2/2+u) \ k_2!},
$$
employing \e{comb3}.
Let $A(u)$ denote the left side of (\ref{summ}) and $B(u)$ the right side.  Now $ A(u) \G(u)/\G(2u)$ must be a rational function of $u$ and similarly for $B(u)$. Therefore
$$
\frac{\G(u)}{\G(2u)} A(u)=\frac{A_1(u)}{A_2(u)}, \qquad \frac{\G(u)}{\G(2u)} B(u)=\frac{B_1(u)}{B_2(u)}
$$
for polynomials $A_1,A_2,B_1,B_2$. We have demonstrated that
\begin{equation}\label{poly}
A_1(u)B_2(u) = A_2(u) B_1(u)
\end{equation}
for infinitely many integers $u$. Hence the left and right sides of (\ref{poly}) are identical and $A(u)=B(u)$ for all $u \in \C$, as required. 
\end{proof}
This completes the proof of Proposition \ref{RankinS}.
\end{proof}

A similar inner product to $\langle f, y^{-k/2}E^*_{k_1}(\cdot, \overline u)
E^*_{k_2}(\cdot, \overline v) \rangle$ for the group $\G(N)$ is computed in a different way in \cite[Prop. 2.5]{KR17}. This result is used there to prove relations among products of Eisenstein series that mirror the Manin relations.

 The $L$-function of $f$ is
$
L(f,s) := \sum_{n=1}^\infty a_f(n) n^{-s},
$
defined for $\Re(s)$ large.  The completed $L$-function,
\begin{equation}\label{lfs2}
L^*(f,s):=(2\pi)^{-s}\G(s) L(f,s) = \int_0^\infty f(iy) y^{s-1} \, dy,
\end{equation}
is now an analytic function for all $s \in \C$.
Suppose $f$ is a  Hecke eigenform that is normalized to have $a_f(1)=1$. Then the convolution $L$-series from \e{olym} satisfies
\begin{equation*}
  \sum_{n=1}^\infty \frac{a_f(n) \sigma_{2v-1}(n)}{n^w} = L(f,s) L(f,w)/\zeta(2u),
\end{equation*}
for $\Re(w)$ large enough, by comparing Euler products as in \cite[Eq. (2.11)]{DO10}. Therefore Proposition \ref{RankinS} and analytic continuation give the next corollary which is
\cite[Prop. 2.1]{DO10}.

\begin{cor} \label{iprodprop}
Let $k_1$, $k_2 \gqs 0$ be even with $k=k_1+k_2$ and $f$ a normalized Hecke eigenform of weight $k$. Then for all $u,v \in \C$ we have the following relation, with $s$ and $w$ given by \e{sw},
\begin{equation}\label{ziggg}
 (-1)^{k_2/2} \Bigl\langle f, y^{-k/2}E^*_{k_1}(\cdot, \overline u)E^*_{k_2}(\cdot, \overline v)\Bigr\rangle
= 2 \cdot \pi^{k/2} L^*(f,s) L^*(f,w).
\end{equation}
\end{cor}


\section{The kernel for products of $L$-functions} \label{proj}
Let $S_k$ be the $\C$-vector space of holomorphic cusp forms of weight $k$. We may choose a basis $\mathcal B_k$ of normalized Hecke eigenforms. For every $s,w \in \C$, the condition
\begin{equation*}
  \s{f}{H_{\overline s, \overline w}} = L^*(f,s) L^*(f,w) \qquad \text{for all} \qquad f \in \mathcal B_k,
\end{equation*}
uniquely defines the kernel $H_{s,w}$ as a cusp form in $S_k$.

Suppose we have $n \in \Z_{\gqs 0}$ and even $k_1,k_2 \gqs 4$   so that $k=k_1+k_2+2n$. Zagier in \cite[Sect. 5]{Za77} gave an explicit description of $H_{n+1,n+k_2}$ in terms of the Rankin-Cohen bracket $[E_{k_1}, E_{k_2}]_n$ of two Eisenstein series:
\begin{equation}\label{qwet}
(-1)^{k_1/2} 2^{3-k} \frac{k_1 k_2}{B_{k_1} B_{k_2}}\binom{k-2}{n} H_{n+1,n+k_2}=\frac{[E_{k_1}, E_{k_2}]_n}{(2\pi i)^n}
\end{equation}
where
\begin{equation} \label{rco}
[E_{k_1}, E_{k_2}]_n := \sum_{r=0}^n (-1)^r \binom{k_1+n-1}{n-r}
\binom{k_2+n-1}{r}E_{k_1}^{(r)}E_{k_2}^{(n-r)}.
\end{equation}
(The $n=0$ case of \e{qwet} is due to Rankin.) The Fourier coefficients of $E_{k_1}, E_{k_2}$ are rational and so it follows from \e{qwet} that the Fourier coefficients of $H_{n+1,n+k_2}$ are rational also. This is the key step used in \cite[Sect. 5]{Za77} and \cite[p. 202]{KZ84} to prove

\begin{theorem}[Manin's Periods Theorem]  \label{per}
For each $f \in \mathcal B_k$ there exist real numbers $\omega_+(f), \ \omega_-(f)$ with $\omega_+(f) \omega_-(f) = \bigl\langle f, f \bigr\rangle$ and
$$
L^*(f,s)/\omega_+(f), \quad L^*(f,w)/\omega_-(f) \in K_f
$$
for all $s,w$ with $1 \lqs s, w \lqs k-1$ and $s$ even, $w$ odd. Here $K_f$ is the finite extension of $\Q$ obtained by adjoining all the Fourier coefficients of $f$.
\end{theorem}

For this see also the discussion in \cite[Sect. 4.3]{DO10}.
In \cite{DO10} we showed another way to demonstrate the rationality of the Fourier coefficients of $H_{s,w}$ as summarized next. Suppose that $F:\H \to \C$ is smooth, transforms with holomorphic weight $k$ and has at most polynomial growth at $\ci$. Then there exists a unique cusp form in $S_k$, which we label $\pi_{hol}(F)$ so that
\begin{equation} \label{fms}
  \s{g}{F} = \s{g}{\pi_{hol}(F)}  \qquad \text{for all} \qquad g \in S_k.
\end{equation}
The cusp form $\pi_{hol}(F)$ is called the {\em holomorphic projection} of $F$. If $F(z) =O(y^{-\epsilon})$ as $y\to \infty$ for some $\epsilon>0$ and has Fourier expansion $F(z)=\sum_{\ell \in \Z} F_\ell(y) e^{2\pi i \ell x}$, then letting $g$ in \e{fms} be the $\ell$th Poincar\'e series and unfolding gives the formula
\begin{equation} \label{hope}
  \frac{(4\pi \ell)^{k-1}}{(k-2)!} \int_0^\infty F_\ell(y) e^{-2 \pi \ell y} y^{k-2} \, dy
\end{equation}
for the $\ell$th Fourier coefficient of $\pi_{hol}(F)$.

It follows from Corollary \ref{iprodprop} that
\begin{equation*}
  H_{s,w}= \frac 1{2 \cdot \pi^{k/2}} \pi_{hol}\left( (-1)^{k_2/2} y^{-k/2}E^*_{k_1}(\cdot, \overline u)E^*_{k_2}(\cdot, \overline v)\right)
\end{equation*}
for all $s,w \in \C$ where, by \e{sw},
$
  2u=s+w-k+1$ and $2v=-s+w+1$.
Therefore, we may find the Fourier coefficients of $H_{s,w}$ from those of $E^*_{k_1}(\cdot, \overline u)E^*_{k_2}(\cdot, \overline v)$ by the formula \e{hope}. The case when $s$ and $w$ are integers of opposite parity in the range $1\lqs s,w\lqs k-1$ corresponds exactly, by \cite[Lemma 3.2]{DO10}, to $u$ and $v$ being integers for which there exist positive even $k_1,k_2$ where $k_1+k_2=k$,
\begin{equation}
1-k_1/2 \lqs u \lqs k_1/2 \quad \text{and} \quad
1-k_2/2 \lqs v \lqs k_2/2.
\label{d7}
\end{equation}
Hence $u^*<k_1/2$ and $E^*_{k_1}(z,u)$ is in the upper triangle of Figure \ref{bfig}, having only positive Fourier coefficients. The same is true of $E^*_{k_2}(z,v)$ and so the Fourier coefficients of the product $E^*_{k_1}(z,u)E^*_{k_2}(z,v)$ are finite sums with terms given by Theorem \ref{main2}. In section 3 of \cite{DO10} this calculation is carried out and  \e{hope}  applied to find the $\ell$th Fourier coefficient of $H_{s,w}$. Precisely, this $\ell$th  coefficient is given by  the finite formula on the right of \cite[Eq. (1.12)]{DO10} divided by $2^{2-k}(k-2)!$. In particular it is rational as required for the proof of Theorem \ref{per}.

A further interesting expression for  $H_{s,w}(z)$ is found in \cite{DO13} in terms of the `double Eisenstein series'
\begin{equation}\label{dbleis3}
\es_{s, k-s}(z,w):=
\sum_{\substack{\g, \, \delta \in B \backslash \G \\  c_{\g  \delta^{-1}}  >0}}
\left( c_{\g  \delta^{-1}} \right)^{w-1} \left( \frac{j(\g,z)}{j(\delta,z)} \right)^{-s}j(\delta,z)^{-k}
\end{equation}
where $B:=\left\{(\smallmatrix 1
& n \\ 0 & 1 \endsmallmatrix ) : n\in \Z\right\}$ and  $c_{\g  \delta^{-1}}$ indicates the bottom left entry of the matrix $\g  \delta^{-1}$. This series  converges for $2<\Re(s)<k-2$ and $\Re(w)<\Re(s)-1, k-1-\Re(s)$ to a holomorphic cusp form  of weight $k$; see Section 4 of \cite{DO13}. By \cite[Thm. 2.3]{DO13},
\begin{equation*}
H_{s,w}(z) =\left[\frac{  e^{s i \pi/2} \G(s) \G(k-s) \G(k-w)\zeta(1-w+s)\zeta(1-w+k-s)}{2^{3-w}\pi^{k+1-w}   \G(k-1) }\right] \es_{s, k-s}(z,w).
\end{equation*}

We close with one more example of the holomorphic projection of a product. For $q=e^{2\pi i z}$, let $\Delta(z):=q\prod_{n=1}^\infty (1-q^n)^{24}$  be the discriminant function in $S_{12}$. Its Fourier coefficients are given by Ramanujan's tau function  $\tau(n)$. Recall the harmonic Eisenstein series $\ei_{-2}(z)$ of weight $-2$ with Fourier coefficients given in Theorem \ref{andr}.  The weight $10$ cusp form $\pi_{hol}(\Delta \cdot \ei_{-2})$ must be identically zero. Hence its $n$th  Fourier coefficient is zero and we obtain by \e{hope}, for all $n\in \Z_{\gqs 1}$,
\begin{multline} \label{fnl}
  \tau(n)\left(\zeta(3)-\frac{11}{16 n^3}\right)
  = -\sum_{m=1}^{n-1}  \tau(n-m) \sigma_{-3}(m)\\ - \sum_{m=1}^{\infty}  \tau(m+n) \sigma_{-3}(m)
\sum_{u=0}^2 \binom{u+8}{8} \frac{m^u n^9}{(m+n)^{u+9}}.
\end{multline}
This formula \e{fnl} may be compared with the more rapidly converging series for $\zeta(3)$ in \e{z3}, \e{lerchxxx}, \e{z3b} and \e{lerchxx}.

{\small
\bibliography{eisen}

\begin{thebibliography}{BFOR17}

\bibitem[AD]{AD}
Nickolas Andersen and William Duke.
\newblock Modular invariants for real quadratic fields and {K}loosterman sums.
\newblock {\em arXiv:1801.08174}.

\bibitem[ALR]{ALR}
Nickolas Andersen, Jeffrey~C. Lagarias, and Robert~C. Rhoades.
\newblock Shifted polyharmonic {M}aass forms for {$\rm{PSL}(2,\mathbb Z)$}.
\newblock {\em arXiv:1708.01278}.

\bibitem[Ber77]{Be77}
Bruce~C. Berndt.
\newblock Modular transformations and generalizations of several formulae of
  {R}amanujan.
\newblock {\em Rocky Mountain J. Math.}, 7(1):147--189, 1977.

\bibitem[Ber89]{Be89}
Bruce~C. Berndt.
\newblock {\em Ramanujan's notebooks. {P}art {II}}.
\newblock Springer-Verlag, New York, 1989.

\bibitem[BFOR17]{BFOR}
Kathrin Bringmann, Amanda Folsom, Ken Ono, and Larry Rolen.
\newblock {\em Harmonic {M}aass forms and mock modular forms: theory and
  applications}, volume~64 of {\em American Mathematical Society Colloquium
  Publications}.
\newblock American Mathematical Society, Providence, RI, 2017.

\bibitem[BK18]{BK}
Kathrin Bringmann and Stephen Kudla.
\newblock A classification of harmonic {M}aass forms.
\newblock {\em Math. Ann.}, 370(3-4):1729--1758, 2018.

\bibitem[Bro18a]{B18}
Francis Brown.
\newblock A class of non-holomorphic modular forms {I}.
\newblock {\em Res. Math. Sci.}, 5:5:7, 2018.

\bibitem[Bro18b]{B}
Francis Brown.
\newblock A class of non-holomorphic modular forms {III}: real analytic cusp
  forms for {${\rm SL}_2(\mathbb Z)$}.
\newblock {\em Res. Math. Sci.}, 5(3):5:34, 2018.

\bibitem[BS17]{BS17}
Bruce~C. Berndt and Armin Straub.
\newblock Ramanujan's formula for {$\zeta(2n+1)$}.
\newblock In {\em Exploring the {R}iemann zeta function}, pages 13--34.
  Springer, Cham, 2017.

\bibitem[Bum97]{Bu}
Daniel Bump.
\newblock {\em Automorphic forms and representations}, volume~55 of {\em
  Cambridge Studies in Advanced Mathematics}.
\newblock Cambridge University Press, Cambridge, 1997.

\bibitem[CJK10]{CJK10}
Gautam Chinta, Jay Jorgenson, and Anders Karlsson.
\newblock Zeta functions, heat kernels, and spectral asymptotics on
  degenerating families of discrete tori.
\newblock {\em Nagoya Math. J.}, 198:121--172, 2010.

\bibitem[CS17]{CS17}
Henri Cohen and Fredrik Str\"omberg.
\newblock {\em Modular forms, A classical approach}, volume 179 of {\em
  Graduate Studies in Mathematics}.
\newblock American Mathematical Society, Providence, RI, 2017.

\bibitem[dAP00]{Pr00}
Wladimir de~Azevedo~Pribitkin.
\newblock Eisenstein series and {E}ichler integrals.
\newblock In {\em Analysis, geometry, number theory: the mathematics of {L}eon
  {E}hrenpreis ({P}hiladelphia, {PA}, 1998)}, volume 251 of {\em Contemp.
  Math.}, pages 463--467. Amer. Math. Soc., Providence, RI, 2000.

\bibitem[DD18]{DD}
Eric D'Hoker and William Duke.
\newblock Fourier series of modular graph functions.
\newblock {\em J. Number Theory}, 192:1--36, 2018.

\bibitem[DFI02]{DFI02}
William Duke, John~B. Friedlander, and Henryk Iwaniec.
\newblock The subconvexity problem for {A}rtin {$L$}-functions.
\newblock {\em Invent. Math.}, 149(3):489--577, 2002.

\bibitem[DIT18]{DIT}
William Duke, \"Ozlem Imamo\=glu, and \'Arp\'ad T\'oth.
\newblock Kronecker's first limit formula, revisited.
\newblock {\em Res. Math. Sci.}, 5(2):Paper No. 20, 21, 2018.

\bibitem[{\relax DLMF}]{DLMF}
{\it NIST Digital Library of Mathematical Functions}.
\newblock http://dlmf.nist.gov/, Release 1.0.17 of 2017-12-22.
\newblock F.~W.~J. Olver, A.~B. {Olde Daalhuis}, D.~W. Lozier, B.~I. Schneider,
  R.~F. Boisvert, C.~W. Clark, B.~R. Miller and B.~V. Saunders, eds.

\bibitem[DO10]{DO10}
Nikolaos Diamantis and Cormac O'Sullivan.
\newblock Kernels of {$L$}-functions of cusp forms.
\newblock {\em Math. Ann.}, 346(4):897--929, 2010.

\bibitem[DO13]{DO13}
Nikolaos Diamantis and Cormac O'Sullivan.
\newblock Kernels for products of {$L$}-functions.
\newblock {\em Algebra Number Theory}, 7(8):1883--1917, 2013.

\bibitem[GMR11]{Mur11}
Sanoli Gun, M.~Ram Murty, and Purusottam Rath.
\newblock Transcendental values of certain {E}ichler integrals.
\newblock {\em Bull. Lond. Math. Soc.}, 43(5):939--952, 2011.

\bibitem[Gro70]{Gr70}
Emil Grosswald.
\newblock Die {W}erte der {R}iemannschen {Z}etafunktion an ungeraden
  {A}rgumentstellen.
\newblock {\em Nachr. Akad. Wiss. G\"ottingen Math.-Phys. Kl. II}, 1970:9--13,
  1970.

\bibitem[Gro72]{Gr72}
Emil Grosswald.
\newblock Comments on some formulae of {R}amanujan.
\newblock {\em Acta Arith.}, 21:25--34, 1972.

\bibitem[Gro75]{Gr75}
Emil Grosswald.
\newblock Rational valued series of exponentials and divisor functions.
\newblock {\em Pacific J. Math.}, 60(1):111--114, 1975.

\bibitem[Iwa02]{IwSp}
Henryk Iwaniec.
\newblock {\em Spectral methods of automorphic forms}, volume~53 of {\em
  Graduate Studies in Mathematics}.
\newblock American Mathematical Society, Providence, RI, second edition, 2002.

\bibitem[Jak94]{Ja94}
Dmitry Jakobson.
\newblock Quantum unique ergodicity for {E}isenstein series on {${\rm
  PSL}_2({\bf Z})\backslash {\rm PSL}_2({\bf R})$}.
\newblock {\em Ann. Inst. Fourier (Grenoble)}, 44(5):1477--1504, 1994.

\bibitem[KMR17]{KR17}
Kamal Khuri-Makdisi and Wissam Raji.
\newblock Periods of modular forms and identities between {E}isenstein series.
\newblock {\em Math. Ann.}, 367(1-2):165--183, 2017.

\bibitem[KN09]{KN09}
Masanori Katsurada and Takumi Noda.
\newblock Differential actions on the asymptotic expansions of non-holomorphic
  {E}isenstein series.
\newblock {\em Int. J. Number Theory}, 5(6):1061--1088, 2009.

\bibitem[KZ84]{KZ84}
Winfried Kohnen and Don Zagier.
\newblock Modular forms with rational periods.
\newblock In {\em Modular forms ({D}urham, 1983)}, Ellis Horwood Ser. Math.
  Appl.: Statist. Oper. Res., pages 197--249. Horwood, Chichester, 1984.

\bibitem[LR16]{LR16}
Jeffrey~C. Lagarias and Robert~C. Rhoades.
\newblock Polyharmonic {M}aass forms for {$\rm{PSL}(2,\mathbb Z)$}.
\newblock {\em Ramanujan J.}, 41(1-3):191--232, 2016.

\bibitem[Maa83]{Maa}
Hans Maass.
\newblock {\em Lectures on modular functions of one complex variable},
  volume~29 of {\em Tata Institute of Fundamental Research Lectures on
  Mathematics and Physics}.
\newblock Tata Institute of Fundamental Research, Bombay, second edition, 1983.
\newblock With notes by Sunder Lal.

\bibitem[Miy06]{Miy}
Toshitsune Miyake.
\newblock {\em Modular forms}.
\newblock Springer Monographs in Mathematics. Springer-Verlag, Berlin, english
  edition, 2006.
\newblock Translated from the 1976 Japanese original by Yoshitaka Maeda.

\bibitem[MSW11]{MSW11}
M.~Ram Murty, Chris Smyth, and Rob~J. Wang.
\newblock Zeros of {R}amanujan polynomials.
\newblock {\em J. Ramanujan Math. Soc.}, 26(1):107--125, 2011.

\bibitem[Ono09]{Ono}
Ken Ono.
\newblock Unearthing the visions of a master: harmonic {M}aass forms and number
  theory.
\newblock In {\em Current developments in mathematics, 2008}, pages 347--454.
  Int. Press, Somerville, MA, 2009.

\bibitem[O'S02]{O02}
Cormac O'Sullivan.
\newblock Identities from the holomorphic projection of modular forms.
\newblock In {\em Number theory for the millennium, {III} ({U}rbana, {IL},
  2000)}, pages 87--106. A K Peters, Natick, MA, 2002.

\bibitem[O'S16]{O16}
Cormac O'Sullivan.
\newblock Zeros of the dilogarithm.
\newblock {\em Math. Comp.}, 85(302):2967--2993, 2016.

\bibitem[Ter76]{Te76}
Audrey Terras.
\newblock Some formulas for the {R}iemann zeta function at odd integer argument
  resulting from {F}ourier expansions of the {E}pstein zeta function.
\newblock {\em Acta Arith.}, 29(2):181--189, 1976.

\bibitem[WW96]{WW}
E.~T. Whittaker and G.~N. Watson.
\newblock {\em A course of modern analysis}.
\newblock Cambridge Mathematical Library. Cambridge University Press,
  Cambridge, 1996.
\newblock Reprint of the fourth (1927) edition.

\bibitem[Zag77]{Za77}
Don Zagier.
\newblock Modular forms whose {F}ourier coefficients involve zeta-functions of
  quadratic fields.
\newblock pages 105--169. Lecture Notes in Math., Vol. 627, 1977.

\bibitem[Zag91]{Za91}
Don Zagier.
\newblock Periods of modular forms and {J}acobi theta functions.
\newblock {\em Invent. Math.}, 104(3):449--465, 1991.

\bibitem[Zag08]{Za}
Don Zagier.
\newblock Elliptic modular forms and their applications.
\newblock In {\em The 1-2-3 of modular forms}, Universitext, pages 1--103.
  Springer, Berlin, 2008.

\end{thebibliography}
}

{\small 
\vskip 5mm
\noindent
\textsc{Dept. of Math, The CUNY Graduate Center, 365 Fifth Avenue, New York, NY 10016-4309, U.S.A.}

\noindent
{\em E-mail address:} \texttt{cosullivan@gc.cuny.edu}
}

\end{document}